\documentclass[12pt,leqno,twoside]{article}
\usepackage{amssymb}
\usepackage{amsmath}
\usepackage{amsthm}
\usepackage{t1enc}
\usepackage[cp1250]{inputenc}
\usepackage{a4,indentfirst,latexsym}
\usepackage{graphics}
\usepackage{mathrsfs}
\usepackage[normalem]{ulem}
\usepackage{cite,enumitem,graphicx}
\usepackage[colorlinks=true,urlcolor=black,
citecolor=black,linkcolor=black,linktocpage,pdfpagelabels,
bookmarksnumbered,bookmarksopen]{hyperref}
\usepackage[colorinlistoftodos]{todonotes}

\input xy
\xyoption{all}

\newtheorem{Th}{Theorem}[section]
\newtheorem{Prop}[Th]{Proposition}
\newtheorem{Lem}[Th]{Lemma}
\newtheorem{Cor}[Th]{Corollary}

\newtheorem{Rem}[Th]{Remark}

\newenvironment{altproof}[1]
{\noindent%\addvspace{0.3cm}
{\em Proof of {#1}}.}
{\nopagebreak\mbox{}\hfill $\Box$\par\addvspace{0.5cm}}

\newcommand{\wt}{\widetilde}

   \newcommand{\vp}{\varphi}
   
   \newcommand{\eps}{\varepsilon}

   \def\div{\mathop{\mathrm{div}\,}}

   \def\id{\mathrm{id}}
    \def\O{\mathrm{O}}
   
   \def\Z{\mathbb{Z}}

   \def\N{\mathbb{N}}
   \def\R{\mathbb{R}}

   \def\curl{\mathrm{curl}}

   \def\cl{\mathrm{cl\,}}

   \def\V{\mathcal{V}}

   \def\W{\mathcal{W}}
   
   \def\C{\mathbb{C}}

\newcommand{\cB}{{\mathcal B}}
\newcommand{\cC}{{\mathcal C}}
\newcommand{\cD}{{\mathcal D}}
\newcommand{\cE}{{\mathcal E}}

\newcommand{\cH}{{\mathcal H}}

\newcommand{\cK}{{\mathcal K}}

\newcommand{\cM}{{\mathcal M}}
\newcommand{\cN}{{\mathcal N}}

\newcommand{\cP}{{\mathcal P}}

\newcommand{\cT}{{\mathcal T}}

\newcommand{\cV}{{\mathcal V}}
\newcommand{\cW}{{\mathcal W}}

\newcommand{\fI}{I}%{{\mathfrak I}}
\newcommand{\fJ}{J}%{{\mathfrak J}}

\newcommand{\al}{\alpha}
\newcommand{\be}{\beta}
\newcommand{\ga}{\gamma}
\newcommand{\de}{\delta}
\newcommand{\la}{\lambda}
\newcommand{\si}{\sigma}
\newcommand{\om}{\omega}

\newcommand{\Ga}{\Gamma}
\newcommand{\Om}{\Omega}
\newcommand{\Si}{\Sigma}

\def\curlop{\nabla\times}
\newcommand{\weakto}{\rightharpoonup}
\newcommand{\pa}{\partial}
\def\id{\mathrm{id}}

\newcommand{\tX}{\widetilde{X}}
\newcommand{\tu}{\widetilde{u}}
\newcommand{\tv}{\widetilde{v}}
\newcommand{\tcV}{\widetilde{\cV}}

\newcommand{\cTto}{\stackrel{\cT}{\longrightarrow}}

\numberwithin{equation}{section}

\begin{document}
\title{Nonlinear time-harmonic Maxwell equations in an anisotropic bounded medium}
\author{Thomas Bartsch \and Jaros\l aw Mederski
\footnote{The author was partially supported by the National Science Centre, Poland (Grant No. 2014/15/D/ST1/03638).}
}
\date{}
\maketitle

\begin{abstract}
  We find solutions $E:\Om\to\R^3$ of the problem
  \[
  \left\{
  \begin{aligned}
  &\curlop(\mu(x)^{-1}\curlop E) - \om^2\eps(x) E = \pa_E F(x,E) &&\quad \text{in }\Om\\
  &\nu\times E = 0 &&\quad \text{on }\pa\Om
  \end{aligned}
  \right.
  \]
  on a bounded Lipschitz domain $\Om\subset\R^3$ with exterior normal $\nu:\pa\Om\to\R^3$. Here $\curlop$ denotes the curl operator in $\R^3$. The equation describes the propagation of the time-harmonic electric field $\Re\{E(x)e^{i\om t}\}$ in an anisotropic material with a magnetic permeability tensor $\mu(x)\in\R^{3\times3}$ and a permittivity tensor $\eps(x)\in\R^{3\times3}$. The boundary conditions are those for $\Om$ surrounded by a perfect conductor. It is required that $\mu(x)$ and $\eps(x)$ are symmetric and positive definite uniformly for $x\in\Om$, and that $\mu,\eps\in L^{\infty}(\Om,\R^{3\times 3})$. The nonlinearity $F:\Om\times\R^3\to\R$ is superquadratic and subcritical in $E$, the model nonlinearity being of Kerr-type: $F(x,E)=|\Ga(x)E|^p$ for some $2<p<6$ with $\Ga(x)\in GL(3)$ invertible for every $x\in\Om$ and $\Ga,\Ga^{-1}\in L^\infty(\Om, \R^{3\times 3})$. We prove the existence of a ground state solution and of bound states if $F$ is even in $E$. Moreover if the material is uniaxial we
find two types of solutions with cylindrical symmetries.
\end{abstract}

{\bf MSC 2010:} Primary: 35Q60; Secondary: 35J20, 58E05, 78A25

{\bf Key words:} time-harmonic Maxwell equations in anisotropic nonlinear media; uniaxial media; ground state; variational methods for strongly indefinite functionals

\section{Introduction}\label{sec:intro}

The paper is concerned with electromagnetic waves in an anisotropic, inhomogeneous and nonlinear medium $\Om$ in the absence of charges, currents and magnetization. In such a medium the constitutive relations between the electric displacement field $\cD$ and the electric field $\cE$ as well as between the magnetic induction $\cH$ and the magnetic field $\cB$ are given by
\[%begin{equation}\label{eq:relations}
\cD=\eps\cE +\cP_{NL} \quad\text{and}\quad \cB=\mu \cH,
\]%end{equation}
where $\eps$ is the (linear) permittivity tensor of the anisotropic material, and $\cP_{NL}$
%=(\cP_{NL1},\cP_{NL2},\cP_{NL3})$$
stands for the nonlinear polarization. In anisotropic and inhomogeneous media $\eps$ depends on $x\in\Om$, and $\cP_{NL}$ depends on the direction of the vector $\cE=(\cE_1,\cE_2,\cE_3)$ and on $x\in\Om$. The permittivity tensor $\eps(x)\in\R^{3\times3}$ and the permeability tensor $\mu(x)\in\R^{3\times3}$ are assumed to be symmetric and uniformly positive definite for $x\in\Om$. The Maxwell equations
\begin{equation*}\label{eq:Maxwell}
\left\{
\begin{aligned}
    &\curlop \cH = \pa_t \cD, \quad \div(\cD)=0,\\
    &\pa_t \cB + \curlop \cE=0, \quad\div(\cB)=0,
\end{aligned}
\right.
\end{equation*}
together with the constitutive relations lead to the equation (see Saleh and Teich \cite{FundPhotonics})
$$
\curlop\left(\mu(x)^{-1}\curlop \cE\right)+\eps\partial_t^2 \cE = -\partial_t^2 \cP_{NL}.
$$
In the time-harmonic case the fields $\cE$ and $\cP$ are of the form
$\cE(x,t) = \Re\{E(x)e^{i\omega t}\}$, $\cP_{NL}(x,t) = \Re\{P(x)e^{i\omega t}\}$, with $E(x),P(x)\in\C^3$, so we arrive at the time-harmonic Maxwell equation
\begin{equation}\label{eq:main}
\curlop\left(\mu(x)^{-1}\curlop E\right) - V(x)E = f(x,E) \qquad\textnormal{in } \Om,
\end{equation}
where  $V(x) = \omega^2\eps(x)$ and $f(x,E)$ takes care of the nonlinear polarization. We consider nonlinearities of the form $f(x,E)=\pa_E F(x,E)$. In Kerr-like media one has
$$
F(x,E)=|\Gamma(x)E|^4
$$
with $\Ga(x)\in GL(3)$ invertible for every $x\in\Om$ and $\Ga,\Ga^{-1}\in L^{\infty}(\Om,\R^{3\times 3})$. This will be our model nonlinearity but we shall consider more general nonlinearities; see Section \ref{sec:results}.

The goal of this paper is to find solutions $E:\Om\to\R^3$ of \eqref{eq:main} together with the boundary condition
\begin{equation}\label{eq:bc}
\nu\times E = 0\qquad\text{on }\pa\Om
\end{equation}
where $\nu:\pa\Om\to\R^3$ is the exterior normal. This boundary condition holds when $\Om$ is surrounded by a perfect conductor.

Solutions of \eqref{eq:main} are critical points of the functional
\begin{equation}\label{eq:action}
\fJ(E) = \frac12\int_\Om\langle\mu(x)^{-1}\curlop E,\curlop E\rangle\, dx
          - \frac12\int_\Om \langle V(x)E, E\rangle\, dx - \int_\Om F(x,E)\,dx
\end{equation}
defined on an appropriate subspace $W^p_0(\curl;\Om)$ of $H_0(\curl;\Om)$; see Section~\ref{sec:results} for the definition of the spaces we work with. In the spirit of the Helmholtz decomposition any $E\in W^p_0(\curl;\Om)$ can be written as $E=v+w$ with $w$ irrotational, i.e.\ $\curlop w=0$, and $\div(V(x)v)=0$. The functional has the form
\begin{equation*}%\label{eq:action2}
\fJ(v+ w)
 = \frac12\int_\Om\langle\mu(x)^{-1}\curlop v,\curlop v\rangle\,dx
    - \frac12\int_\Om \langle V(x)(v+ w),v+ w\rangle\,dx - \int_\Om F(x,v+w)\,dx.
\end{equation*}
This functional is unbounded from above and from below, the curl operator has an infinite-dimensional kernel, and critical points have infinite Morse index. Although $\fJ$ has a linking geometry in the spirit of Benci and Rabinowitz \cite{BenciRabinowitz}, the problem cannot be treated by standard variational methods as in \cite{BenciRabinowitz,BartschDing,DingBook} due to a lack of compactness. The derivative $\fJ':W^p_0(\curl;\Om)\to \big(W^p_0(\curl;\Om)\big)^*$ is not weak-weak$^*$ continuous even when the growth of $F$ is subcritical.

In the literature there are only few results about nonlinear equations like \eqref{eq:main} involving the curl-curl operator. If $\Om=\R^3$ then Benci and Fortunato \cite{BenFor} proposed, within a unified field theory for classical electrodynamics, the equation
\begin{equation}\label{eq:BF}
\curlop\curlop A = W'(|A|)A
\end{equation}
for the gauge potential $A$ related to the magnetic field $H=\curlop A$. Azzollini et al.\ \cite{BenForAzzAprile} and D'Aprile and Siciliano \cite{DAprileSiciliano} used the symmetry of the domain $\R^3$ and of \eqref{eq:BF} in order to find special types of symmetric solutions. Symmetry also plays an important role in the paper \cite{Bartsch:2014} by Bartsch et al.\ which is concerned with the isotropic case on $\Om=\R^3$ where $\mu$ and $V$ are scalar, $F(x,E)=\Gamma(x)|E|^p$, $2<p<6$, with $V$ and $F$ being cylindrically symmetric, say functions of $\sqrt{x_1^2+x_2^2}$ and $x_3$, and periodic in $x_3$-direction. Mederski \cite{MederskiENZ} considered \eqref{eq:main} on $\Om=\R^3$ with $\mu$ being scalar and assuming that $V\in L^q(\R^3)$ for several values of $q$ which depend on the growth of $F(x,u)$ as $u\to0$ and $|u|\to\infty$. In \cite{MederskiENZ} it is also required that $F$ is $\Z^3$-periodic in $x$, not cylindrically symmetric. Cylindrically symmetric media have also been considered in the work of Stuart and Zhou \cite{StuartZhou96}--\cite{StuartZhou10} on transverse electric and transverse magnetic solutions. The search for these solutions reduces to a one-dimensional variational problem or an ODE, which simplifies the problem considerably.

We would also like to mention that linear time-harmonic Maxwell equations have been extensively studied by means of numerical and analytical methods, on bounded and unbounded (exterior) domains; see e.g.\ \cite{Ball2012,BuffaAmmariNed,Leis68,Picard01,KirschHettlich,Monk,Doerfler} and the references therein.

Equation \eqref{eq:main} in the nonsymmetric case and on a bounded domain has first been been studied by the authors in \cite{BartschMederski1} where we developed a critical point theory in order to find ground states and bound states for \eqref{eq:main}. There $\Om$ was required to be simply connected with connected $\cC^{1,1}$ boundary, hence diffeomorphic to the unit ball in $\R^3$. Moreover $\mu$ and $V$ had to be scalar and constant, i.e.\ only the isotropic case has been treated in \cite{BartschMederski1}. Concerning the nonlinearity a structural condition had to be assumed that is difficult to check even for sums of Kerr type nonlinearities. In the present paper we significantly improve the results from \cite{BartschMederski1} in several ways. In particular, there will be no restrictions on the topology of $\Om$, and we allow $\mu$ and $V$ to be non-isotropic tensors. Moreover, in an axisymmetric setting we also obtain the existence of solutions as in \cite{DAprileSiciliano} which has not been considered in \cite{BartschMederski1}. In addition, we are able to deal with nonlinearities that cannot be treated with the methods of \cite{BartschMederski1}. For instance we can allow that $F(x,E)=0$ if $|E|$ is small, modelling the case that the Kerr effect is linear for low intensities of the electric field $\cE$. We are also able to weaken or even to get rid of the severe structural restriction on $F$ mentioned above. In order to achieve this we refine the Nehari-Pankov manifold technique used in \cite{BartschMederski1}, obtain more careful estimates, and we introduce a new approach in a setting where the Nehari-Pankov manifold does not exist.

\section{Statement of results}\label{sec:results}

Throughout the paper we assume that $\Om\subset\R^3$ is a bounded domain with Lipschitz boundary. We begin with recalling the basic spaces in which we look for solutions of \eqref{eq:main}. The space
$$
H(\curl;\Om) := \{E\in L^2(\Om,\R^3): \curlop E \in L^2(\Om,\R^3)\}
$$
is a Hilbert space when provided with the graph norm
$$
\|E\|_{H(\curl;\Om)} := \left(|E|^2_2+|\curlop E|^2_2\right)^{1/2}.
$$
Here and in the sequel $|\cdot|_q$ denotes the $L^q$-norm. The curl of $E$, $\curlop E$, has to be understood in the distributional sense. The closure of $\cC^{\infty}_0(\Om,\R^3)$ in $H(\curl;\Om)$ is denoted by $H_0(\curl;\Om)$. There is a continuous tangential trace operator $\gamma_t:H(\curl;\Om)\to H^{-1/2}(\pa\Om)$ such that
$$
\gamma_t(E)=\nu\times E_{|\pa\Om}\qquad\text{for any $E\in \cC^{\infty}(\overline\Om,\R^3)$}
$$
and (see \cite[Theorem~3.33]{Monk})
$$
H_0(\curl;\Om)=\{E\in H(\curl;\Om): \gamma_t(E)=0\}.
$$
We also need the space
\[
\cV = \left\{v\in H_0(\curl;\Om): \int_\Om\langle V(x)v,\varphi\rangle\,dx=0
        \text{ for every $\varphi\in \cC^\infty_0(\Om,\R^3)$ with $\curlop\varphi=0$} \right\}.
\]
%Here $\curlop w$ has to be understood in the distributional sense.

Now we state our hypotheses on the linear part of \eqref{eq:main}.
\begin{itemize}
\item[(L1)] $\mu,V\in L^{\infty}(\Omega,\R^{3\times 3})$, and $\mu(x),V(x)$ are symmetric and uniformly positive definite for $x\in\Om$.
\item[(L2)] $\cV$ is compactly embedded into $L^p(\Om,\R^3)$ for some $2<p<6$.
\end{itemize}
In the next section we present conditions on $\cV$ which imply (L2). An important role plays the curl-curl source eigenvalue problem
\begin{equation}\label{EgEigenvalue}
\left\{
\begin{array}{ll}
\curlop(\mu(x)^{-1}\curlop u) = \la V(x)u,\ \ \div(V(x)u)=0
&
\hbox{in } \Om,\\
\nu\times u =0
&
\hbox{on } \pa\Om.
\end{array}
\right.
\end{equation}
We need in particular the eigenspace for $\la=1$, i.e.\ the kernel of the operator $\curlop(\mu(x)^{-1}\curlop)-V(x)$ in $\cV$:
\[
\cV_0:=\{v\in\cV: v\text{ solves \eqref{EgEigenvalue} for }\la=1\}.
\]

Concerning the nonlinearity $f(x,E)=\pa_EF(x,E)$ we collect various assumptions that we shall use. The model nonlinearity $F(x,E)=|\Ga(x)E|^p$ with $2<p<6$ as in (L2) satisfies all hypotheses provided $\Ga(x)\in GL(3)$ and $\Ga,\Ga^{-1}\in L^\infty(\Om,\R^{3\times3})$. In applications, for low intensity $|\cE|$ of the electric field $\cE$, the Kerr effect is often considered to be linear, i.e.\ $\cP_{NL}=0$ for small $|\cE|$ (see \cite{MederskiENZ}). In order to model also this nonlinear phenomenon we consider nonlinearities of the form
\begin{itemize}
\item[(F0)] $F(x,u)=F_0(x,\chi(u))$ with
$\displaystyle
\chi(u)=
 \begin{cases}
 0 &\text{if } |u|\le\de,\\
 \left(1-\frac{\de}{|u|}\right)u &\text{if } |u|>\de,
 \end{cases}
$\qquad
for some $\de\ge0$.
\end{itemize}
Now we state our conditions on $F_0$.
\begin{itemize}
\item[(F1)] $F_0:\Om\times\R^3\to\R$ is differentiable with respect to $u\in\R^3$, such that $f_0=\pa_uF_0:\Om\times\R^3\to\R^3$ is a Carath\'eodory function (i.e.\ measurable in $x\in\Om$, continuous in $u\in\R^3$ for a.e.\ $x\in\Om$). Moreover, $F_0(x,0)=0$ for a.e.\ $x\in\Om$.
\item[(F2)] $|f_0(x,u)|=o(|u|)$ as $u\to0$ uniformly in $x\in\Om$.
\item[(F3)] There exists a constant $c>0$ such that
    \[|f_0(x,u)|\le c(1+|u|^{p-1})\qquad\text{for all } x\in\Om, u\in\R^3.\]
\end{itemize}
Observe that (F1)-(F3) also hold for $F$ as in (F0). These conditions are standard and yield in particular that solutions of \eqref{eq:main}, \eqref{eq:bc} can be obtained with variational methods. The next condition describes the growth of $F_0$ as $|u|\to\infty$.
\begin{itemize}
\item[(F4)] $F_0(x,u)\geq 0$ for $x\in\Om$, $u\in\R^3$ and there exists a constant $d>0$, such that
    $$
    \liminf_{|u|\to\infty}F_0(x,u)/|u|^p>d>0\hbox{ uniformly with respect to } x\in\Om.
    $$
\end{itemize}
The remaining conditions are of a structural nature. The next condition allows to introduce the Nehari-Pankov manifold and to define a ground state as minimizer of the energy functional on this manifold which has infinite dimension and infinite co-dimension). In order to formulate it we introduce the function
\[
\vp(t,x,u,v)
 := \frac{t^2-1}{2}\langle f_0(x,u),u\rangle + t\langle f_0(x,u),v\rangle + F_0(x,u)
     - F_0(x,tu+v)
\]
defined for $t\ge0$, $x\in\Om$, $u,v\in\R^3$.
\begin{itemize}
\item[(F5)] (i) For a.e.\ $x\in\Om$ and for all $t\ge0$, $u,v\in\R^3$ there holds $\vp(t,x,u,v)\le0$.

    (ii) For $t\ge0$, $u\in L^p(\Om)$ and $v\in\cV_0$ with $tu+v\ne u$ there holds $\int_\Om\vp(t,x,u,v)\,dx<0$.
\end{itemize}
The integral condition in (F5)(ii) is like a Landesman-Lazer condition which is used in asymptotically linear elliptic problems when the linearization at infinity has a kernel. It implies the following convexity condition for $F$ which is needed for the semicontinuity of the associated energy functional and for the linking geometry of $J$.

\begin{itemize}
\item[(F6)]
    \begin{itemize}
    \item[(i)] $F_0(x,u)$ is convex with respect to $u\in\R^3$ for a.e.\ $x\in\Om$.
    \item[(ii)] For every $u\in L^p(\Om)$ the functional
    \[
    \cV_0\to\R,\quad v\mapsto \int_\Om F_0(x,u+v)\,dx,
    \]
    is strictly convex.
    \end{itemize}
\end{itemize}

\begin{Rem}\label{rem:convexity} a) In order to see that (F5) implies (F6) fix $x\in\Om$, $u_0,u_1\in\R^3$, and consider the map $g(s):=F_0(x,(1-s)u_0+su_1)$. Then (F5)(i) with $t=1$, $u=(1-s)u_0+su_1$, $v=(t-s)(u_1-u_0)$ gives for $0\le s<t\le 1$:
\[
\begin{aligned}
g'(s)(t-s) &= \langle f_0((1-s)u_0+su_1),(t-s)(u_1-u_0)\rangle\\
 &\le F_0((1-t)u_0+tu_1) - F_0((1-s)u_0+su_1) = g(t)-g(s)
\end{aligned}
\]
This implies the convexity of $g$, hence (F6)(i). Similarly one sees that (F5) implies (F6)(ii).

b) Condition (F5)(i) also implies that
\[
\langle f_0(x,u),u\rangle \ge 2F_0(x,u)
\]
for a.e.\ $x\in\Om$ and every $u\in\R^3$. Simply set $t=0$ and $v=0$ in (F5)(i).

c) Of course (F6) holds if $F_0$ is strictly convex in $u$ for a.e.\ $x\in\Om$. If (F6)(i) holds and $F_0(x,u)$ is strictly convex in $u$ for $x\in\Om_0$, $\Om_0\subset\Om$ some nonempty open subset, then (F6)(ii) follows provided the unique continuation principle for the time harmonic Maxwell equation $\curlop(\mu(x)^{-1}\curlop u) - V(x)u = 0$ holds. This is the case for large classes of potentials $V$ (see \cite{Okaji:2002, Vogelsang:1991}).

d) In \cite{BartschMederski1} we required the condition
\begin{itemize}
\item[(*)] If $\langle f(x,u),v\rangle = \langle f(x,v),u\rangle \neq 0$ then
    $$
    2(F(x,u)-F(x,v))\langle f(x,u),u\rangle
     \le \langle f(x,u),u\rangle^2-\langle f(x,u),v\rangle^2.
    $$
    If in addition $F(x,u)\ne F(x,v)$ then the strict inequality holds.
\end{itemize}
This condition is difficult to check and not needed any more.
%It was used in \cite{BartschMederski1} to prove a somewhat stronger version of (F5).
%It implies that the map {\color{blue} $\langle f(x,tu),u\rangle/t$} \todo{small f, not F} is {\color{blue} strictly} increasing in $t>0$ {\color{blue} provided that $\langle f(x,u),u\rangle\neq 0$}  which is necessary for setting up the classical Nehari manifold.
\end{Rem}

If (F5) does not hold we require the following condition of Ambrosetti-Rabinowitz type.
\begin{itemize}
\item[(F7)] $F=F_0$ and there is $\gamma>2$ such that
    $\langle f_0(x,u),u\rangle \geq \ga F_0(x,u)$ for $u\in\R^3$.
\end{itemize}
We obtain solutions of our problem if (F1)-(F4), (F6)-(F7) hold. However, although we require $F=F_0$ it is possible that there exists a sequence of solutions $E_n$ with positive energy $J(E_n)\to0$, hence there may not exist a ground state as in the case of (F0)-(F5). By a ground state we mean a solution $E$ with positive energy $J(E)>0$ that has the least energy among all solutions with positive energy. Observe that if $\de>0$ in (F0) and if $\cV_0\ne\{0\}$ then any $E\in\cV_0$ with $|E|_\infty\le\de$ is a solution $E$ with $J(E)=0$. In order to obtain a ground state the following assumption will prove to be sufficient.
\begin{itemize}
\item[(F8)] There is $\eta\geq\gamma$ such that
    $\eta F_0(x,u)\geq \langle f_0(x,u),u\rangle>0$ for $u\in\R^3\setminus\{0\}$.
\end{itemize}

In order to state our results we introduce the space
$$
W^p(\curl;\Om) := \{E\in L^p(\Om,\R^3): \curlop E\in L^2(\Om,\R^3)\}\subset H(\curl;\Om)
$$
which is a Banach space if provided with the norm
$$
\|E\|_{W^p(\curl;\Om)} := \left(|E|^2_p+|\curlop E|^2_2\right)^{1/2}.
$$
We shall look for solutions of \eqref{eq:main} in the closure
$W^p_0(\curl;\Om)\subset H_0(\curl;\Om)$ of $\cC^{\infty}_0(\Om,\R^3)$ in $W^p(\curl;\Om)$. Observe that $\cV$ is a closed linear subspace of $W^p_0(\curl;\Om)$ as a consequence of (L2). Moreover, since for every $\varphi\in\cC^\infty_0(\Om;\R^3)$ the linear map
\[
E \mapsto \int_\Om \langle E,\curlop\varphi\rangle dx
\]
is continuous on $W^p_0(\curl;\Om)\subset H(\curl;\Om)$, the space
\[
\begin{aligned}
\cW
 &= \left\{w\in W^p_0(\curl;\Om):\int_\Om\langle w,\curlop\varphi\rangle = 0
     \text{ for all }\varphi\in\cC^\infty_0(\Om,\R^3)\right\}\\
 &= \{w\in W^p_0(\curl;\Om): \curlop w=0\}
\end{aligned}
\]
is a closed complement of $\cV$ in $W^p_0(\curl;\Om)$, hence there is a Helmholtz type decomposition $W^p_0(\curl;\Om) = \cV\oplus\cW$. Helmholtz decompositions hold in very general settings, even in higher dimensions and for exterior domains; see \cite{Pauly:2014} for recent results and references to the literature. Our first main result reads as follows.

\begin{Th}\label{thm:main}
Suppose (L1)-(L2) and (F0)-(F4) hold.

a) If (F5) holds then \eqref{eq:main} has a ground state solution $E\in W^p_0(\curl;\Om)$.

b) If (F6)-(F7) hold then \eqref{eq:main} has a nontrivial solution $E\in W^p_0(\curl;\Om)$. This is a ground state if also (F8) holds.

c) If (F5) or (F6)-(F7) hold, and if $F$ is even in $u$ then \eqref{eq:main} has a sequence of solutions $E_n$ with $J(E_n)\to\infty$.
\end{Th}

If (F5) holds then the ground state solution can be characterized as the minimizer of $J$ on the Nehari-Pankov manifold $\cN$ which has infinite dimension and infinite co-dimension. If (F5) does not hold but (F6)-(F7) do, then we first prove the existence of a nontrivial solution by a mountain pass argument on a constraint $\cM\subset W^p_0(\curl;\Om)$. Afterwards we show that $\inf\{J(E): J(E)>0,\, J'(E)=0\}$ is achieved provided that (F8) is additionally satisfied. If (F5)-(F7) hold then $\cN$ is a submanifold of $\cM$ with co-dimension 1, and the mountain pass argument on $\cM$ gives the minimum of $J$ on $\cN$. The manifolds $\cM$ and $\cN$ will be defined in Section~\ref{sec:CriticalTheory} in an abstract setting, and in Section~\ref{sec:proof-main} for the functional $J$. Note that we can deal with a much wider range of nonlinearities than those considered in \cite{BartschMederski1}.

\begin{Rem}
If $E=v+w\in W^p_0(\curl;\Om)$ is a nontrivial solution of \eqref{eq:main} with $v\in\cV$ and $w\in\cW$ then necessarily $v\ne0$. This is a simple consequence of (L1) and (F6)(i). In fact, testing \eqref{eq:main} with $v+w$ yields: $\curlop(\mu(x)^{-1}\curlop v) \ne V(x)v$; see Proposition~\ref{prop:0isolated}.
\end{Rem}

In the next remark we give examples of nonlinearities satisfying our conditions.

\begin{Rem}\label{RemarkEx}
If $\tilde F:[0,+\infty)\to\R$ satisfies the classical Ambrosetti-Rabinowitz condition, then
\begin{equation}\label{Ex3}
F_0(x,u) = \tilde F(|\Ga(x)u|)
\end{equation}
satisfies (F7). Using this one can easily construct many examples of nonlinearities satisfying (F0)-(F4), (F6)-(F7). Observe that (F1)-(F8) are positively linear conditions, i.e.~if $F_0,G_0$ satisfy these conditions then so does $\al F_0+\be G_0$ for any $\al,\be>0$. This is not the case for condition (*) in Remark~\ref{rem:convexity}~d) which is quadratic in $F_0,f_0$. Therefore it is easy to see that
\begin{equation}\label{Ex1}
F_0(x,u)=\sum_{i=1}^m\frac{1}{p_i}|\Ga_i(x)u|^{p_i}
\end{equation}
satisfies (F1)-(F8), provided $2<\ga=p_1\le p_2\le\dots\le p_m=p=\eta<6$, $\Ga_i(x)\in GL(3)$ for a.e.~$x\in\Om$, and $\Ga_i,\Ga_i^{-1}\in L^\infty(\Om,\R^{3\times3})$. Observe that these functions are not radial when $\Ga_i(x)$ is not an orthogonal matrix. In particular, if $p_i=4$ then \eqref{Ex1} models the Kerr-effect. %Nonlinearities of the form could not be dealt with in \cite{BartschMederski1}.
Nonlinearities of the form \eqref{Ex1} have not been dealt with in [6] because it was unclear whether they satisfy the hypothesis (*) from Remark~\ref{rem:convexity}~d). Given the other conditions from Theorem 2.2, it has been observed in \cite[Remark 5.4 (d)]{BartschMederskiSurvey} that a weaker variant of (*) is essentially equivalent to (F9) from \cite{BartschMederskiSurvey}, which is a stronger variant of (F5).
%
%b) In applications, for low intensity $|\cE|$ of the electric field $\cE$, the Kerr effect is often considered to be linear, $\cP_{NL}$ is negligible and therefore we may assume that $\cP_{NL}\sim 0$ for small $|\cE|$ (see \cite{MederskiENZ}). In order to model also this nonlinear phenomenon one may consider nonlinearities of the form
%\begin{equation}\label{Ex2}
%\tilde F(x,u) = F(|\Ga(x)u|)\quad\text{ with } F(r) = \frac{(r^p-a^p)_+}{p}
%\end{equation}
%where $2<p<6$, $a\geq 0$, and one easily verifies (F1)--(F6)(i), (F7) and (F8). This type of nonlinearity is not covered by Theorem~\ref{thm:main} because neither (F5)(ii) nor (F6)(ii) hold. Still we can prove the existence of infinitely many solutions of
%\begin{equation}\label{eq:maintilde}
%\curlop\left(\mu(x)^{-1}\curlop E\right) - V(x)E = \tilde f(x,E) \qquad\textnormal{in } \Om,
%\end{equation}
%with $\tilde f = \pa_u \tilde F$. A more general class of nonlinearities of this type can be constructed as follows. Choose $\de>0$ and define $\chi:\R^3\to\R^3$ by $\chi(u)=0$ if $|u|\le\de$, and $\chi(u)=(1-\frac{\de}{|u|})u$ if $|u|>\de$. For a function $F:\Om\times\R^3\to\R$ set $\tilde F(x,u):=F(x,\chi(u))$. Then Theorem~\ref{thm:main} remains true with \eqref{eq:main} replaced by \eqref{eq:maintilde}, keeping the assumptions on $F$. We shall indicate the necessary changes in our proof.
\end{Rem}

Now we concentrate on nonlinear uniaxial media which are of great importance due to the phenomenon of birefringence and applications in crystallography \cite{FundPhotonics,StuartZhou03,Nie}. Here we require that the problem is symmetric with respect to the cylindrical symmetry group $G=\O(2)\times\{1\}\subset O(3)$:
\begin{itemize}
\item[(S)] $\Om$ is invariant with respect to $G$, and $F_0$ is invariant with respect to the action of $G$ on the $x$- and $u$-variables, i.e.\ $F_0(g_1x,g_2u)=F_0(x,u)$ for all $x\in\Om$, $u\in\R^3$, $g_1,g_2\in G$. Moreover, $\mu(x)$ and $V(x)$ commute with $G$, and $\mu,V$ are invariant with respect to $G$, i.e.\ $g_2\mu(g_1x)g_2^{-1}=\mu(x)$ for all $x\in\Om$, $g_1,g_2\in G$; similarly for $V$.
\end{itemize}
Observe that a symmetric matrix $A$ commutes with $G$ if and only if it is of the form
\begin{equation}\label{eq:formOfeps1}
A=\begin{pmatrix}a & 0 & 0\\0 & a & 0\\0 & 0 & b\end{pmatrix},
\end{equation}
Thus we require that the permeability tensor $\mu$ and the tensor $V$, which corresponds to the permittivity tensor $\eps$, have the form \eqref{eq:formOfeps1} with $a,b\in L^{\infty}(\Om)$ positive, bounded away from $0$, and invariant with respect to the action of $G$ on $\Om$. Hence we allow cylindrically symmetric anisotropic materials. In this setting more can be said about the shape of the solutions. In fact, we can show the existence of solutions of the form
\begin{equation}\label{eq:sym2}
E(x)=\al(r,x_3)\begin{pmatrix}-x_2\\x_1\\0\end{pmatrix},\qquad r=\sqrt{x_1^2+x_2^2},
\end{equation}
and of the form
\begin{equation}\label{eq:sym1}
E(x)=\be(r,x_3)\begin{pmatrix}x_1\\x_2\\0\end{pmatrix}
      +\ga(r,x_3)\begin{pmatrix}0\\0\\1\end{pmatrix}.
\end{equation}

\begin{Th}\label{thm:sym1}
Suppose (L1), (F0)-(F4), and (S) hold.
\begin{itemize}
\item[a)] If $F_0$ is even in $u$ and (F5) or (F7) hold then there exist infinitely many solutions of the form \eqref{eq:sym2} and with positive energy. Moreover there exists a least energy solution among all solutions with positive energy of the form \eqref{eq:sym2} provided (F5) or (F7)-(F8) hold. Every solution of the form \eqref{eq:sym2} is divergence-free and lies in $H^1_0(\Om,\R^3)$.
\item[b)] If (L2) holds and in addition (F5) or (F6)-(F7), then \eqref{eq:main} has a solution $E\in W^p_0(\curl;\Om)$ of the form \eqref{eq:sym1}. Moreover, there exists a least energy solution among all solutions of the form \eqref{eq:sym1} and with positive energy provided (F5) or (F6)-(F8) hold. If $F_0$ is even in $u$, in addition to (L2), (F5) or (F6)-(F7), then \eqref{eq:main} has infinitely many solutions of the form \eqref{eq:sym1} having positive energy.
\end{itemize}
\end{Th}

If (F5) holds then the least energy solutions in Theorem~\ref{thm:sym1} can be obtained by minimization on the Nehari-Pankov manifold in the space of fields of the form \eqref{eq:sym2} or \eqref{eq:sym1}, respectively. Observe that in Theorem~\ref{thm:sym1}~a) we do not assume (L2) nor (F6) since we will be able to restrict our functional to fields of the form \eqref{eq:sym2} which are divergence free and continuously embedded in $H^1_0(\Om,\R^3)$; see Lemma~\ref{LemDecomposition}. This restriction requires the additional symmetry that $F$ is even in $u$. Without this condition we do not know whether a single solution of the form \eqref{eq:sym2} exists.

Even in the isotropic case $\mu=\mu_0\id_{3\times 3}$, $V(x)=\lambda\id_{3\times 3}$, theorems \ref{thm:main} and \ref{thm:sym1} extend results from
\cite[Theorem 2.2 and Theorem 2.3]{BartschMederski1}. The solutions of the form \eqref{eq:sym1} have not been considered in \cite{BartschMederski1}. For $\Om=\R^3$ solutions of the form \eqref{eq:sym2} have been treated in \cite{BenForAzzAprile}, solutions of the form \eqref{eq:sym1} in \cite{DAprileSiciliano}.

\section{Preliminaries}\label{sec:prelim}

As a consequence of (L1) the inner product
\[
(E_1,E_2) = \int_{\Om}\langle \mu(x)^{-1}\curlop E_1,\curlop E_2\rangle
             + \langle V(x)E_1, E_2\rangle\,dx
\]
in $H_0(\curl,\Om)$ is equivalent to the standard inner product in $H(\curl;\Om)$. For $v\in\cV$ and $w\in\cW$ there holds:
\begin{equation}\label{eq:VperpW}
(v,w) = \int_{\Om}\langle V(x)v,w\rangle \,dx = 0
\end{equation}
so $\cV$ and $\cW$ are orthogonal with respect to $(\cdot,\cdot)$. Clearly, $\cW$ contains all gradient vector fields: $\nabla W^{1,p}(\Om)\subset\cW$, hence
\[
\begin{aligned}
\cV &\subset \{E\in W^p_0(\curl;\Om):\div(V(x)E)=0\}\\
&\subset\left\{E\in H_0(\curl;\Om): \div(V(x)E)\in L^2(\Om,\R^3)\right\} =: X_N(\Om,V).
\end{aligned}
\]
Therefore assumption (L2) holds in particular if $X_N(\Om,V)$ embeds into $H^1(\Om,\R^3)$. This has been proved in \cite[Theorem~2.12]{Amrouche} for $V=\id_{3\times 3}$ and $\partial \Om$ of class $\cC^{1,1}$. Costabel et al.\ \cite{CostabelDN1999} and Hiptmair \cite[Section 4]{Hiptmair} obtained the embedding for Lipschitz domains admitting singularities and for isotropic and piecewise constant $V$. The following proposition contains another setting when (L2) holds.

\begin{Prop}\label{Prop:embedding} Suppose (L1) holds, $V$ is Lipschitz continuous, and $\Om$ has $\cC^2$ boundary. Then $X_N(\Om,V)$ is continuously embedded in $H^1(\Om,\R^3)$. In particular (L2) holds.
\end{Prop}

\begin{proof}
Any $E \in H_0(\curl;\Om)$ has a standard Helmholtz decomposition $E = u+\nabla w$ with
$u\in\{E\in H_0(\curl;\Om): \div (E)=0\}$ and $w\in H^1_0(\Omega)$. Since $X_N(\Om,\id_{3\times 3})$ is embedded in $H^1(\Om,\R^3)$ there holds $u\in H^1(\Om,\R^3)$. Observe that $w$ solves the divergence form elliptic equation
\begin{equation*}
\div(V(x)\nabla w) = \div(f),\quad w\in H^1_0(\Om),
\end{equation*}
with $f = V(x)E-V(x)u$. As a consequence of $u\in H^1(\Om,\R^3)$ and $V \in W^{1,\infty}(\Om,\R^3)$ we obtain $\div(V(x)u) \in L^2(\Om)$, hence $\div(f) \in L^2(\Om)$. The operator
$L := \div(V(x)\nabla(\cdot))$ is strictly elliptic and therefore $w\in H^2(\Om)$ by
\cite[Theorem 8.12]{GilbargTrudinger}. This implies $E = u+\nabla w \in H^1(\Om,\R^3)$.
\end{proof}

Note that $\cV$ is a Hilbert space with the scalar product
\begin{equation}\label{eq:scal-prod-on-cV}
\langle u,v \rangle_\cV := \int_{\Om}\langle \mu(x)^{-1}\curlop u,\curlop v\rangle\; dx.
\end{equation}
If $\Omega$ is simply connected with connected boundary, then the {\em normal cohomology space}
$$
K_N(\Om)=\{E\in H_0(\curl;\Om):\; \curlop E=0,\;\div(E)=0\}
$$
is trivial and $\W=\nabla W^{1,p}_0(\Omega)$. This is the case considered in  \cite{BartschMederski1}. The spectrum of the curl-curl operator in $H_0(\curl;\Om)$ consists of the eigenvalue $0$ with infinite multiplicity and eigenspace $\nabla H^1_0(\Om)$, and of a sequence of positive eigenvalues with finite multiplicities and eigenfunctions in
$\{v\in H_0(\curl;\Om):\div(v)=0\}$; see \cite[Corollary 3.51, Theorem~4.18]{Monk}. For a general domain $K_N(\Om)$ is nontrivial and contained in $\W$. We set
\[
\cW_2:=\{w\in H_0(\curl;\Om): \curlop w=0\}.
\]
In the anisotropic situation we investigate the following {\em curl-curl source problem} instead of the spectrum of the curl-curl operator.

\begin{Prop} Suppose (L1) and (L2) hold. Then for any $g\in L^2(\Om,\R^3)$ the equation
\begin{equation}\label{eq:operator}
\curlop(\mu(x)^{-1}\curlop v) + V(x)w = V(x)g
\end{equation}
has a unique solution $(v,w) \in \cV\times\W_2$ and the operator
\[
K:L^2(\Om,\R^3) \to \cV \subset L^2(\Om,\R^3),\quad
 Kg=v\text{ satisfies \eqref{eq:operator} for some }w\in \W_2,
\]
is compact. The restriction $K_\cV:\cV\to\cV$ of $K$ is compact and self-adjoint with respect to the scalar product \eqref{eq:scal-prod-on-cV}.
\end{Prop}

\begin{proof}
The existence and uniqueness of the solution follow from the Babuska-Brezzi theorem; see e.g.\  \cite[Theorem 2.1.4]{Doerfler}. The compactness of $K$, and of $K_\cV$, is a consequence of the compactness of the embedding $\V\hookrightarrow L^2(\Om,\R^3)$. The self-adjointness of $K_\cV$ follows from $\langle Kg,h\rangle_\cV = \int_\Om\langle V(x)g(x),h(x)\rangle_{\R^3}\,dx$ for $g,h\in\cV$.
\end{proof}

\begin{Cor}\label{Cor:Spectrum}
There is a discrete sequence $0<\la_1<\la_2<\la_3<\ldots$ of (anisotropic) Maxwell eigenvalues with eigenspaces of finite multiplicity, i.e.\
$$
\curlop(\mu(x)^{-1}\curlop v) = \la V(x)v
$$
has a solution $v\in\V$ if and only if $\la=\la_k$ for some $k\geq 1$, and the space of solutions is finite-dimensional.
\end{Cor}

\begin{proof}
Observe that if \eqref{eq:operator} holds for some $g=\la v$, then $\la > 0$ and $w = 0$ (cf.\ \cite[Theorem 2.1.7]{Doerfler}).
\end{proof}

From now on we always assume that (L1)-(L2), (F0)-(F6) are satisfied. Then the functional $\fJ: W_0^p(\curl;\Om) \to \R$ given by
\[
\fJ(E) := \frac12\int_\Om\langle\mu(x)^{-1}\curlop E,\curlop E\rangle\,dx
            - \frac12\int_\Om \langle V(x)E,E\rangle\,dx - \int_\Om F(x,E)\,dx
\]
is well defined. For $E=v+w$ with $v\in\cV$ and $w\in\cW$ there holds
\[
\begin{aligned}
\fJ(v+w)
 &= \frac12\int_{\Om}\langle\mu(x)^{-1}\curlop v,\curlop v\rangle\,dx
      - \frac12\int_\Om \langle V(x)v,v\rangle + \langle V(x)w,w\rangle\,dx\\
 &\hspace{1cm}
      - \int_{\Om}F(x,v+w)\,dx.
\end{aligned}
\]
This functional is of class $\cC^1$ with
\[
\begin{aligned}
\fJ'(v+w)(\phi+\psi)
 &= \int_\Om\langle \mu(x)^{-1}\curlop v,\curlop \phi\rangle\; dx
      - \int_\Om(\langle V(x) v,\phi\rangle+\langle V(x)w,\psi\rangle)\; dx\\
 &\hspace{1cm}
      - \int_\Om\langle f(x,v+w),\phi+\psi\rangle \; dx
\end{aligned}
\]
for any $v,\phi\in\cV$ and any $w,\psi\in\cW$. We shall use the following norm in
$W_0^p(\curl;\Om) = \cV\oplus\cW$:
$$
\|v+w\| = \left(\|v\|_\V^2+\|w\|^2_\W\right)^{1/2}
 := \left(\langle\mu(x)^{-1}\curlop v,\curlop v\rangle_{L^2} + |w|^2_p\right)^{1/2}
 \qquad\text{for }v\in\cV,\ w\in\cW
$$
so that
\[
\fJ(v+w) = \frac12\|v\|_\cV^2 - \frac12\int_\Om \langle V(x)(v+w),v+w\rangle\,dx
           - \int_{\Om}F(x,v+w)\,dx,
\]
We can now formulate the variational approach to \eqref{eq:main}.

\begin{Prop}\label{PropSolutE}
$E = v+w \in W_0^p(\curl;\Om) = \cV \oplus \cW$ is a critical point of $\fJ$ if and only if it is a solution of \eqref{eq:main}.
\end{Prop}

\begin{Prop}\label{prop:0isolated}
Suppose the assumptions of  Theorem~\ref{thm:main} a) or b) or c) hold.

a) If $ E= v+w \in\cV \oplus \cW$ is a solution of \eqref{eq:main} with $J(E)>0$ then $\curlop(\mu(x)^{-1}\curlop v) \ne V(x)v$, in particular $v\ne0$.

b) If (F5) or (F6)-(F8) hold then the nontrivial critical values of $J$ are positive and bounded away from 0.%, in particular $0\in W_0^p(\curl;\Om)$ is an isolated solution of \eqref{eq:main}.
\end{Prop}

\begin{proof}
a) Suppose the claim is wrong so that $-V(x)w=f(x,v+w)$ holds. Testing this with $E=v+w$ and using \eqref{eq:VperpW}, (L1), (F1), (F6)(i), we are led to
\begin{equation}\label{eq:test}
0\ge -\int_\Om \langle V(x)w,w\rangle\,dx = \int_\Om \langle f(x,v+w),v+w\rangle\,dx \ge 0\,.
\end{equation}
This implies $w=0$ and $\int_\Om \langle f(x,v),v\rangle\,dx=0$. As a consequence of (F1), (F6)(i) this is possible only if $f(x,v)=0$ for a.e.\ $x\in\Om$, and $\int_\Om F(x,v)\,dx=0$. Then $v\in\cV_0$ and $J(E)=J(v)=0$, a contradiction.

b) This is postponed to Section 5 because we need to work out the appropriate tools.
\end{proof}

\section{Critical point theory on natural constraints}\label{sec:CriticalTheory}

Firstly we recall the critical point theory and the Nehari-Pankov manifold from \cite{BartschMederski1}. Let $X$ be a reflexive Banach space with norm $\|\cdot\|$ and with a topological direct sum decomposition $X = X^+\oplus \tX$, where $X^+$ is a Hilbert space with a scalar product. For $u \in X$ we denote by $u^+ \in X^+$ and $\tu  \in \tX$ the corresponding summands so that $u = u^++\tu $. We may assume that $\langle u,u \rangle = \|u\|^2$ for any
$u\in X^+$ and that $\|u\|^2 = \|u^+\|^2+\|\tu \|^2$. The topology $\cT$ on $X$ is defined as the product of the norm topology in $X^+$ and the weak topology in $\tX$. Thus $u_n \cTto u$ is equivalent to $u_n^+ \to u^+$ and $\tu_n \weakto \tu $.

Let $J \in \cC^1(X,\R)$ be a functional on $X$ of the form
\begin{equation}\label{EqJ}
J(u) = \frac12\|u^+\|^2-I(u) \quad\text{for $u=u^++\tu \in X^+\oplus \tX$}.
\end{equation}
We define the set
\begin{equation}\label{eq:NehariDef}
\cN := \{u\in X\setminus\{0\}: J'(u)|_{\R u\oplus \tX}=0,\ J(u)>0\}
\end{equation}
and suppose the following assumptions hold:
\begin{itemize}
\item[(A1)] $I\in\cC^1(X,\R)$ and $I(u)\ge I(0)=0$ for any $u\in X$.
\item[(A2)] $I$ is $\cT$-sequentially lower semicontinuous:
    $u_n\cTto u\quad\Longrightarrow\quad \liminf I(u_n)\ge I(u)$
\item[(A3)] If $u_n\cTto u$ and $I(u_n)\to I(u)$ then $u_n\to u$.
\item[(A4)] There exists $r>0$ such that $a:=\inf\limits_{u\in X^+:\|u\|=r} J(u)>0$.
\item[(B1)] $\|u^+\|+I(u)\to\infty$ as $\|u\|\to\infty$.
\item[(B2)] $I(t_nu_n)/t_n^2\to\infty$ if $t_n\to\infty$ and $u_n^+\to u^+$ for some $u^+\neq 0$ as $n\to\infty$.
\item[(B3)] $\frac{t^2-1}{2}I'(u)[u] + tI'(u)[v] + I(u) - I(tu+v) < 0$ for every $u\in \cN$, $t\ge 0$, $v\in \tX$ such that $u\neq tu+v$.
\end{itemize}

\begin{Prop}\label{prop:nehari}
For every $u \in SX^+ := \{u\in X^+:\|u\|=1\}$ the functional $J$ constrained to
$\R u+\tX = \{tu+v:t\ge0,\ v\in\tX\}$ has precisely two critical points $u_1,u_2$ with positive energy. These are of the form $u_1=t_1u+v_1$, $u_2=t_2u+v_2$ with $t_1>0>t_2$, $v_1,v_2\in\tX$. Moreover, $u_1$ is the unique global maximum of $J|_{\R^+u+\tX}$, and $u_2$ is the unique global maximum of $J|_{\R^-u+\tX}$. Moreover, $u_1$ and $u_2$ depend continuously on $u\in SX^+$.
\end{Prop}

\begin{proof}
Using (A1)-(A4) and (B1)-(B2) one sees that $-J$ is weakly sequentially lower semi-continuous and coercive on $\R u+\tX$, for every $u\in X$. Therefore $J|_{\R^+u+\tX}$ has a global maximum $u_1=t_1u+v_1$, $t_1\ge0$, $v_1\in\tX$. Assumption (A4) implies $J(u_1)\ge a>0$, hence $u_1\notin\tX$, so $u_1$ is a critical point of $J|_{\R u+\tX}$ and $t_1>0$. If $u_0\in\R^+u+\tX$ is any critical point of $J|_{\R u+\tX}$ with $J(u)>0$ then $u_0\in\cN$. Now (B3) implies as in the proof of  \cite[Proposition~4.2]{BartschMederski1} that $u_0$ must be a strict global maximum of $J|_{\R^+u+\tX}$, hence $u_0=u_1$. Using this uniqueness property of $u_1$ it follows easily that $u_1$ depends continuously on $u$. Similarly one obtains $u_2$ as a global maximum of $J|_{\R^-u+\tX}$.
\end{proof}

For $u\in SX^+$ we set $n(u):=u_1$ with $u_1$ from Proposition~\ref{prop:nehari}. Observe that $n(-u)=u_2$ and
\begin{equation}\label{eq:NehariCharact}
\cN = \{u\in X\setminus\tX: J'(u)|_{\R u+\tX}=0,\,J(u)>0\} = \{n(u): u\in SX^+\},
\end{equation}
in particular, $\cN$ is a topological manifold, the Nehari-Pankov manifold. Clearly all critical points of $J$ with $J(u)>0$ lie in $\cN$. Since $J$ is not required to be $\cC^2$ the Nehari-Pankov manifold is just a topological manifold homeomorphic to $SX^+$. The functional $J$ is said to satisfy the {\em $(PS)_c^\cT$-condition in} $\cN$ if every $(PS)_c$-sequence $(u_n)_n$ for the unconstrained functional and such that $u_n\in\cN$ has a subsequence which converges in the $\cT$-topology:
\[
u_n \in \cN,\ J'(u_n) \to 0,\ J(u_n) \to c \qquad\Longrightarrow\qquad
u_n \cTto u\in X\ \text{ along a subsequence.}
\]
The following result is due to \cite{BartschMederski1}.

\begin{Th}\label{ThLink1}
Let $J \in \cC^1(X,\R)$ satisfy (A1)-(A4), (B1)-(B3), set $c_{\cN} = \inf_\cN J$ and let $J$ be coercive on $\cN$, i.e. $J(u)\to\infty$  as $\|u\|\to\infty$ and $u\in\cN$. Then the following holds:
\begin{itemize}
\item[a)] $c_{\cN}\ge a>0$ and $J$ has a $(PS)_{c_{\cN}}$-sequence in $\cN$.
\item[b)] If $J$ satisfies the $(PS)_{c_{\cN}}^\cT$-condition in $\cN$ then $c_{\cN}$ is achieved by a critical point of $J$.
\item[c)] If $J$ satisfies the $(PS)_c^\cT$-condition in $\cN$ for every $c$ and if $J$ is even then it has an unbounded sequence of critical values.
\end{itemize}
\end{Th}

Condition (B3) seems to be very restrictive and not easy to check. A more natural condition employs the convexity of $I$ which in turn will be a consequence of the convexity of $F$. We consider the set
\begin{equation}\label{eq:ConstraintM}
\cM := \{u\in X:\, J'(u)|_{\tX}=0\}=\{u\in X:\, I'(u)|_{\tX}=0\}.
\end{equation}
Observe that the last equality follows from the form of $J$ in \eqref{EqJ}. $\cM$ is a (topological) manifold if the following holds:
\begin{itemize}
\item[(B4)] If $u\in\cM$ then $ I(u)<I(u+v)$ for every $v\in \tX$ with $v\neq0$.
\end{itemize}
Note that, if $I$ is strictly convex, then by (A1)-(A2) we easily see that (B4) is satisfied. Observe that for any $u\in X^+$ there is a unique $m(u)\in\cM$ such that
$m(u)^+=u$. Obviously $m(u)\in\cM$ is the unique global maximum of $J|_{u+\tX}$.

\begin{Rem}\label{rem:cNsubsetcM}
If (B3) and (B4) hold then $\cM\supset\cN$. More precisely, for each $u\in SX^+$ let $t_u>0$ be defined by $n(u)=t_uu+v$ with $v\in\wt X$. Then the map $\be_u:[0,\infty)\to\R$ defined by $\be_u(t)=J(m(tu))$ achieves its maximum at $t_u>0$. If $\be_u'(t)=J'(m(tu))[u]=0$ then $J'(m(tu))|_{\R u\oplus\wt X}=0$, hence $m(tu)\in\cN$ and $t=t_u$. It follows that $\be_u(t)$ is strictly increasing on $[0,t_u]$ and strictly decreasing on $[t_u,\infty)$. Thus $\cN=\{m(t_uu):u\in SX^+\}$ splits $\cM$ into two components:
\[
\cM\setminus\cN=\{m(tu):u\in SX^+,\, 0\le t<t_u\}\cup\{m(tu):u\in SX^+,\, t>t_u\}
\]
\end{Rem}

Our main result of this section reads as follows.

\begin{Th}\label{ThLink2}
Let $J \in \cC^1(X,\R)$ satisfy (A1)--(A4), (B1), (B2), (B4) and set
\begin{equation}
c_{\cM} = \inf_{\gamma\in\Gamma}\sup_{t\in [0,1]} J(\gamma(t))
\end{equation}
where
$$\Gamma:=\{\gamma\in\cC([0,1],\cM):\; \gamma(0)=0,\, \|\ga(1)^+\|>r, \hbox{ and } J(\gamma(1))<0\}.$$
Then the following holds:
\begin{itemize}
\item[a)] $c_{\cM}>0$ and $J$ has a $(PS)_{c_{\cM}}$-sequence in $\cM$.
\item[b)] If $J$ satisfies the $(PS)_{c_{\cM}}^\cT$-condition in $\cM$ then $c_{\cM}$ is achieved by a critical point of $J$.
\item[c)] If $J$ satisfies the $(PS)_c^\cT$-condition in $\cM$ for every $c$ and if $J$ is even then it has an unbounded sequence of critical values.
%\item[d)] If in addition (B3) holds then $c_{\cM}\leq c_{\cN}$, and if $c_{\cM}$ is achieved by a critical point then $c_{\cM}=c_{\cN}$.
\item[d)] If in addition (B3) holds then $c_{\cM}=c_{\cN}$.
\end{itemize}
\end{Th}

\begin{proof}
Recall that for any $u\in X^+$ there is a unique $m(u)\in\cM$ with $m(u)^+=u$. We claim that:
\begin{itemize}
\item[(i)] $m:X^+\to \cM$  is a homeomorphism with inverse $\cM\ni u\mapsto u^+\in X^+$.
\item[(ii)] $J\circ m:X^+\to\R$ is $\cC^1$.
\item[(iii)] $(J\circ m)'(u) = J'(m(u))|_{X^+}:X^+\to\R$ for every $u\in X^+$.
\item[(iv)] $(u_n)_n\subset X^+$ is a Palais-Smale sequence for $J\circ m$ if, and only if, $(m(u_n))_n$ is a Palais-Smale sequence for $J$ in $\cM$.
\item[(v)] $u\in X^+$ is a critical point of $J\circ m$ if, and only if, $m(u)$ is a critical point of $J$.
\item[(vi)] If $J$ is even, then so is $J\circ m$.
\end{itemize}
Now we prove these statements.

(i) Let $u_n\to u_0$ in $X^+$ and $m(u_n)=u_n+v_n$, where $v_n\in \tX$ for all $n\geq 0$. In view of (B4) one has
\begin{equation}\label{eqM:1}
I(m(u_n))\leq I(u_n)\leq I(u_0)+1
\end{equation}
for almost all $n$. Now (B1) implies that $v_n$ is bounded, so we may assume that $v_n\weakto v_0$. As a consequence of (A2) and (B4) we deduce
$$I(m(u_0))\leq I(u_0+v_0)\leq \liminf I(m(u_n))\leq \liminf I(u_n+(m(u_0)-u_0))=I(m(u_0)).$$
Finally, using (A3) and (B4) we obtain $m(u_n)\to m(u_0)=u_0+v_0$.

(ii) Let $u,v\in X^{+}$ and $h\in\R$. Let $m(u+hv)=u+hv+\tu(h)$ for some $\tu(h)\in \tX$.
Observe that by (B4) and by the mean value theorem
\begin{eqnarray*}
I(m(u+hv))-I(m(u))
&\geq& I(u+hv+\tu(h))-I(u+\tu(h))\\
&=& I'(\theta_1(h))(hv)
\end{eqnarray*}
for some $\theta_1(h)\to u+\tu(0)$ as $h\to 0$.
Similarly we have
\begin{eqnarray*}
I(m(u+hv))-I(m(u))
&\leq& I(u+hv+\tu(0))-I(u+\tu(0))\\
&=& I'(\theta_2(h))(hv)
\end{eqnarray*}
for some $\theta_2(h)\to u+\tu(0)$ as $h\to 0$. Thus we obtain
\begin{equation}\label{DefOfJm'}
(I\circ m)'(u)(v) = \lim_{h\to 0}\frac{I(m(u+hv))-I(m(u))}{h} = I'(m(u))(v).
\end{equation}
Using (i) it follows that $(I\circ m)'(u)$ is continuous, therefore $I\circ m$ and $J\circ m$ are of class $\cC^1$ and (ii) holds.

Observe that (iii) follows from $(I\circ m)'(u) = I'(m(u))$ and from the form of $J$ given in \eqref{EqJ}. Finally, (iv), (v) and (vi) are easy consequences of the definition of $m$.

Next we prove that $J\circ m$ has the classical mountain pass geometry. Assumption (A4) implies \begin{equation}\label{eq:mp1}
J\circ m(u) \ge J(u) \ge a > 0 \text{ if } \|u\|=r.
\end{equation}
In order to see for $0\ne u\in X^+$ that
\begin{equation}\label{eq:mp2}
J\circ m(tu) = \frac12\|m(tu)^+\|^2 - I(m(tu)) \to -\infty \text{ as }t\to\infty
\end{equation}
write $m(tu)=tu+\wt{u}_t$ with $\wt{u}_t\in\wt X$, and set
$u_t = u+\frac1t \wt{u}_t = \frac1t m(tu)$. Then
\[
\frac{1}{t^2}I(m(tu)) = \frac{1}{t^2}I(tu_t) \to \infty\quad\text{as }t\to\infty
\]
by (B2). The mountain pass condition \eqref{eq:mp2} follows immediately. Setting
$$
\Si:=\{\si\in\cC([0,1],X^+):\, \si(0)=0,\ \|\si(1)^+\|>r \hbox{ and }
            J\circ m(\si(1))<0\}
$$
the mountain pass value for $J\circ m$ is given by:
$$
c_{\cM}=\inf_{\sigma\in\Sigma}\sup_{t\in [0,1]}J\circ m(\sigma(t)) \ge a > 0.
$$
In view of the mountain pass theorem and using (iv), there exists a $(PS)_{c_{\cM}}$-sequence $(u_n)_n$ for $J$ in $\cM$, which proves a).

In order to prove b) we consider a $(PS)_c$-sequence $(u_n)_n\subset X^+$ for $J\circ m$. Then $(m(u_n))_n$ is a Palais-Smale sequence for $J$ in $\cM$ by (iv), hence $m(u_n)\cTto v$ after passing to a subsequence. This implies $u_n=m(u_n)^+\to v^+$ and we have proved:
\begin{itemize}
\item[(vii)] If $J$ satisfies the $(PS)_c^\cT$-condition in $\cM$ for some $c$ then
   $J\circ m$ satisfies the $(PS)_c$-condition.
\end{itemize}
Next observe that if $J$ satisfies the $(PS)_{c_{\cM}}^\cT$-condition in $\cM$ then $c_{\cM}$ is achieved by a critical point $u\in X^+$ of $J\circ m$, hence $m(u)\in\cM$ is a critical point of $J$ with $J(m(u))=c_{\cM}$. This implies b).

c) follows from the classical symmetric mountain pass theorem. The condition \eqref{eq:mp2} implies that for every finite-dimensional subspace $Y\subset X^+$ there exists $R=R(Y)>0$ such that $J\circ m\le0$ on $Y\setminus B_RY$. Therefore together with \eqref{eq:mp1} and the Palais-Smale condition $J\circ m$ satisfies the hypotheses of \cite[Theorem~9.12]{Rabinowitz:1986}, hence it possesses an unbounded sequence of critical values.

It remains to prove d), so we assume that (B3) holds. Given $u\in\cN$ by \eqref{eq:mp2} there exists $t_0>0$ such that $J(m(t_0u^+))<0$. Therefore the path $\ga(t)=m(tt_0u^+)$, $t\in [0,1]$, lies in $\Ga$. Since $u$ is the unique maximum of $J$ on $\R^+u+\wt X$ there holds $J(\gamma(t))\leq J(u)$, and therefore $c_{\cM}\leq c_{\cN}$. In order to see the reverse inequality observe that Remark~\ref{rem:cNsubsetcM} implies that for any $\ga\in\Ga$ there exists $t\in[0,1]$ with $\ga(t)\in\cN$.
\end{proof}

\section{Proof of Theorem~\ref{thm:main}}\label{sec:proof-main}

We want to find critical points of the functional $\fJ:X:=W_0^p(\curl;\Om)\to\R$ from \eqref{EqJ}. We assume (L1)-(L2), (F0)-(F4), and (F5) or (F6)-(F7). If (F5) holds we shall apply Theorem~\ref{ThLink1}, if (F6)-(F7), may be (F8), hold we shall apply Theorem~\ref{ThLink2}. Recall that (F5) implies (F6), so in the sequel we shall assume (L1)-(L2) and (F0)-(F4) as well as (F6), often without mentioning, but we shall always state when we use (F5),(F7), or (F8).

In order to define $X^+$ and $\tX$ let $0<\la_1<\la_2<\ldots$ be the sequence of eigenvalues (with finite multiplicities) of the curl-curl source problem from Corollary~\ref{Cor:Spectrum}. Let $\cV^+$ be the positive eigenspace of the quadratic form $Q:\cV\to\R$ defined by
$$
Q(v):=\int_\Om \left(\langle\mu(x)^{-1}\curlop v,\curlop v\rangle -
         \langle V(x)v,v\rangle\right)\,dx,
$$
and let $\tcV$ be the semi-negative eigenspace of $Q$. Then $\tcV$ is the finite sum of the eigenspaces associated to all $\la_k\le1$, and $\cV^+$ is the infinite sum of the eigenspaces associated to the eigenvalues $\la_k>1$. Here $\tcV=\{0\}$ if $\la_1>1$, of course. Observe that
\begin{equation}\label{NormInU0}
Q(v)\geq \Big(1-\frac{1}{\lambda_m}\Big)\int_{\Om}\langle\mu(x)^{-1}\curlop v,\curlop v\rangle\,dx
\qquad\text{for any }v\in \cV^+,
\end{equation}
where $m=\min\{k\in\N_0:\la_k>1\}$. If $m\ge2$ and $\la_{m-1}<1$ then
\begin{equation}\label{CondInU1Ker}
Q(v) \leq -\Big(\frac{1}{\lambda_{m-1}}-1\Big)
           \int_{\Om}\langle\mu(x)^{-1}\curlop v,\curlop v\rangle\,dx
\qquad\text{for any }v\in \tcV.
\end{equation}
If $\la_{m-1}=1$ then the kernel of the operator $\curlop(\mu(x)^{-1}\curlop)- V(x)$ is just the eigenspace associated to $\la_{m-1}$. For $v\in\V$ we denote by $v^+\in\V^+$ and $\tv \in\tcV$ the corresponding summands such that $v=v^++\tv$. Now we define $X^+:=\cV^+$ and $\tX:=\tcV\oplus\cW$.

The functional $\fJ:X\to\R$ from Section~\ref{sec:prelim} has the form
\[
\fJ(v+w) = \frac12\|v^+\|^2 - \fI(v+w)
\]
as in \eqref{EqJ} with
\[
\begin{aligned}
\fI(v+w)
 &= -\frac12\|\tv\|_\cV^2 + \frac12\int_\Om\langle V(x)(v+w),v+w\rangle\,dx + \int_\Om F(x,v+w)\\
 &= -\frac12\|\tv\|_\cV^2 + \frac12\int_\Om\langle V(x)v,v\rangle\,dx
     +\frac12\int_{\Om}\langle V(x)w,w\rangle\,dx + \int_\Om F(x,v+w).
\end{aligned}
\]
Now we show that $\fJ$ satisfies the assumptions (A1)-(A4) as well as (B1), (B2), (B4) from Section~\ref{sec:CriticalTheory}. This requires (F0)-(F4) and (F6).

\begin{Lem} If (L1) and (F4) hold then there exists $d'>0$ such that
\begin{equation}\label{eqest3}
\frac12\int_{\Om}\langle V(x)u,u\rangle\,dx+\int_{\Om} F(x,u)\, dx \geq d'|u|^{p}_{p}\quad
\textnormal{ for any } u\in L^p(\Om,\R^3).
\end{equation}
\end{Lem}

\begin{proof}
In view of (F4) we find $M>0$ such that $F(x,u)\geq d |u|^p$ for $|u|>M$. Observe that there is a constant $V_0>0$ such that
\begin{eqnarray*}
\frac12\int_{\Om}\langle V(x)u,u\rangle\,dx+\int_{\Om} F(x,u)\, dx
&\geq& V_0 \int_{\Om}|u|^2\,dx+d\int_{|u|>M}|u|^p\,dx\\
&\geq& V_0 \int_{|u|\leq M}|u|^2\,dx+d\int_{|u|>M}|u|^p\,dx\\
&\geq& d'|u|^p_p
\end{eqnarray*}
where $d'=\min\{V_0 M^{2-p},d\}>0$
\end{proof}

The next lemma shows that (A1)-(A4) and (B1), (B2) hold.

\begin{Lem}\label{LinkingLemma} Suppose (L1)-(L2), (F0)-(F4) and (F6) hold.
\begin{itemize}
\item[\rm a)] $\fI$ is of class $\cC^1$,  $\fI(E)\geq 0$ for any $E\in X$, and $\fI$ is $\cT$-sequentially lower semicontinuous.
\item[\rm b)]  If $E_n\cTto E$ and $I(E_n)\to I(E)$ then $E_n\to E$.
\item[\rm c)] There is $r>0$ such that \ \
%\begin{equation}\label{EqMPG}
$\displaystyle 0<\inf_{\stackrel{v\in \V^+}{\|v\|_\V=r}}\fJ(v)$.
%\end{equation}
\item[\rm d)] $\|v_n^+\|_\V + \fI(v_n+w_n)\to\infty$ as $\|v+w\|\to\infty$.
\item[\rm e)] $\fI(t_n(v_n+w_n))/t_n^2\to\infty$ if $t_n\to\infty$ and $v_n^+\to v_0^+\neq 0$ as $n\to\infty$.
\end{itemize}
\end{Lem}

\begin{proof} a) Since $Q$ is negative semi-definite on $\tcV$ and using (F4) we deduce that $\fI(v+w)\geq 0$ for any $v\in\cV$, $w\in\cW$. The convexity condition (F6) implies that $\fI$ is $\cT$-sequentially lower semicontinuous, and $\fI$ is of class $\cC^1$ as a consequence of (F1)-(F3). Thus we obtain a).

b) Consider $E_n,E\in X$ such that $E_n\cTto E$ and $\fI(E_n)\to\fI(E)$. Writing $E_n=v_n+w_n$, $E=v+w$ with $v_n,v\in\cV$, $w_n,w\in\cW$ we have $v_n^+\to v^+$, $\tv_n\weakto\tv$ in $\cV$, $w_n\weakto w$ in $\cW$. Passing to a subsequence we may assume that $\tv_n\to\tv$ in $\cV$, hence
\[
\begin{aligned}
&\frac12\int_\Om\langle V(x)(v^+_n+w_n),v^+_n+w_n\rangle\,dx + \int_\Om F(x,v_n+w_n)\; dx\\
&\hspace{1cm}
 \to \frac12\int_\Om\langle V(x)(v^++w),v^++w\rangle\,dx + \int_\Om F(x,v+w)\,dx.
\end{aligned}
\]
By the weakly sequentially lower semicontinuity
$$
\int_\Om\langle V(x)(v^+_n+w_n),v^+_n+w_n\rangle\,dx\to
\int_\Om\langle V(x)(v^++w),v^++w\rangle\,dx
$$
and in view of (L1)
\begin{equation}\label{EqconvinL2}
|v^+_n+w_n|_2 \to |v^++w|_2.
\end{equation}
Since $v^+_n+w_n \weakto v^++w$ in $L^p(\Om,\R^3)$ then, up to a subsequence,
$v^+_n+w_n \weakto v^++w$ in $L^2(\Om,\R^3)$, and by (\ref{EqconvinL2}) we have
$v_n^++w_n\to v^++w$ in $L^2(\Om,\R^3)$. Hence
$$
E_n=v_n+w_n\to E=v+w\hbox{ a.e.\ on }\Om.
$$
Finally observe that
\begin{eqnarray*}
&&\int_\Om F(x,E_n) - F(x,E_n-E)\,dx\\
&&\hspace{1cm}
   = \int_\Om\int_0^1\frac{d}{dt}F(x,E_n+(t-1)E)\,dtdx\\
&&\hspace{1cm}
  = \int_0^1\int_\Om\langle f(x,E_n+(t-1)E),E\rangle\,dxdt.
\end{eqnarray*}
Since $f(x,E_n+(t-1)E)\to f(x,tE)$ a.e.\ on $\Om$ Vitali's convergence theorem yields
\begin{eqnarray*}
&&\int_\Om F(x,E_n)-F(x,E_n-E)\,dx\\
&&\hspace{1cm}
  \to \int_0^1\int_\Om\langle f(x,tE),E\rangle\,dxdt
  = \int_\Om F(x,E)\,dx
\end{eqnarray*}
as $n\to\infty$. Moreover, since $\int_\Om F(x,E_n)\to \int_\Om F(x,E)\,dx$ there holds
\begin{equation}\label{EqConvf}
\int_\Om F(x,E_n-E)\,dx \to 0,
\end{equation}
hence $|E_n-E|_p \to 0$ by \eqref{eqest3}, and finally $w_n\to w$ in $L^p(\Om,\R^3)$. This shows (A3).

c) In order to prove c) we observe that assumptions (F0)-(F3) imply that for any $\eps>0$ there is a constant $c_\eps>0$ such that
%\begin{equation}\label{eqest1}
%|f(x,u)|\leq \eps|u| + c_{\eps}|u|^{p-1}\quad \textnormal{ for any } x\in\Om,\;u\in\R^3
%\end{equation}
%and
\begin{equation*}
\int_\Om F(x,u)\, dx \leq \eps|u|^2_{2} + c_\eps|u|^{p}_{p}\quad
\textnormal{ for any } u\in L^p(\Om,\R^3).
\end{equation*}
Using this and \eqref{NormInU0} we deduce for $v\in\V^+$
\begin{eqnarray*}
\fJ(v)
 &=& \frac12Q(v)-\int_\Om F(x,v)\;dx
 \geq \frac\de2\|v\|_\V^2 -\eps|v|_2- c_\eps|v|_p^p\\
&\geq& \frac\de4\|v\|_\V^2 - C_1\|v\|_\V^p
\end{eqnarray*}
for some constant $\delta,C_1>0$ which proves c).

d) Consider a sequence $(v_n+w_n)_n$ in $X$ such that $\|v_n+w_n\|\to\infty$ as $n\to\infty$ and $(\|v_n^+\|_\cV)_{n}$ is bounded. Then $\|\tv_n+w_n\|^2 = \|\tv_n\|_\cV^2+|w_n|_p^2\to\infty$, hence $|\tv_n+w_n|_p\to\infty$ because $\tcV$ is finite-dimensional. Using (L1), the orthogonality $\cV^+\perp\tcV$, $\cV\perp\cW$ with respect to $(\cdot,\cdot)$ and the H\"older inequality we deduce
\begin{equation}\label{EqOrthU_0}
\|\tv_n\|_\V^2\leq (\tv_n,\tv_n)\leq (v_n+w_n,v_n+w_n)\leq C_1|v_n+w_n|_2^2
 \leq C_2|v_n+w_n|_p^2
\end{equation}
for some constants $C_1,C_2>0$. Now \eqref{eqest3} implies
\[
\fI(v_n+w_n) \geq -\frac{C_2}{2}|v_n+w_n|_p^2 + d'|v_n+w_n|_p^p \to \infty,
\]
and d) follows.

e) Consider sequences $t_n\to\infty$ and $v_n\in\cV$, $w_n\in\cW$ such that $v_n^+\to v_0^+\ne0$ as $n\to\infty$. Note that by \eqref{eqest3} and \eqref{EqOrthU_0}
\[
\fI(t_n(v_n+w_n))/t_n^2 \geq -\frac{C_2}{2}|v_n+w_n|_p^2 +d' t_n^{p-2}|v_n+w_n|_p^p.
\]
If $\|v_n+w_n\|\to\infty$ as $n\to\infty$ then $|v_n+w_n|_p\to\infty$, hence
\begin{equation}\label{eq:LemCoer_d}
\fI(t_n(v_n+w_n))/t_n^2\to \infty
\end{equation}
and we are done. Now suppose $(\|v_n+w_n\|)_n$ is bounded, hence $(|v_n+w_n|_p)_n$ is bounded. If $|v_n+w_n|_p\to0$  then $|v_n+w_n|_2\to 0$ which implies $v_n^+\to 0$ in $L^2(\Om,\R^3)$ contradicting $v_0\neq0$. Therefore $t_n^{p-2}|v_n+w_n|_p\to\infty$ as $n\to\infty$ and again \eqref{eq:LemCoer_d} holds.
\end{proof}

As in Section~\ref{sec:CriticalTheory} we define
\[
\begin{aligned}
\cM &:= \{E\in X: \, \fJ'(E)[\phi+\psi]=0\hbox{ for any }\phi\in\tcV,\,\psi\in\cW\}.
\end{aligned}
\]

\begin{Lem}\label{LemB4check}  Suppose (L1)-(L2), (F0)-(F4) and (F6) hold.

a) $I$ is strictly convex.

b) (B4) holds.
\end{Lem}

\begin{proof}
a) Observe that if $m=\min\{k\in\N_0:\la_k>1\}\geq 2$ and $\lambda_{m-1}<1$, then $\cV_0=0$ and $-Q$ is strictly convex on $\tcV$. If $\lambda_{m-1}=1$ then $\int_\Om F(x,v)dx$ is strictly convex on the set of all $v\in\cV_0$ by (F6), using also $F=F_0$ from (F7). Finally if $m=1$ then $\tcV=\{0\}$, hence $X=\V^+\oplus \W$ and
$X \ni E \mapsto \int_\Om\langle V(x)E,E\rangle\,dx \in \R$ is strictly convex. Therefore in all cases we obtain that
$$
I(v+w)=-\frac{1}{2}Q(\tv)+ \frac12\int_\Om\langle V(x)(v^++w),v^++w\rangle\,dx
      +\int_\Om F(x,v+w)
$$
is strictly convex.

b) follows from the strict convexity of $I$.
\end{proof}

\begin{Lem}\label{LemB3check}
If (F5) is satisfied then condition (B3) holds.
\end{Lem}

\begin{proof}
Let $E\in\cN$, $t\geq 0$, $\phi\in\tcV$, $\psi\in\W$ satisfy $E \ne tE+\phi+\psi$. We need to show that
\begin{equation}\label{eq:B3check}
\begin{aligned}
&\fI'(E)\left[\frac{t^2-1}{2}E+t(\phi+\psi)\right] + \fI(E) - \fI(tE+\phi+\psi)\\
%&&\hspace{1cm}
% =\frac12\|\phi\|_\V^2 + \frac\la2|\phi+\psi|_2^2 + \int_\O\vp(t,x)\;dx\\
&\hspace{1cm}
 = \frac12Q(\phi) - \frac12\int_\Om\langle V(x)\psi,\psi\rangle\,dx
     + \int_\Om\vp(t,x)\,dx
 < 0
\end{aligned}
\end{equation}
where
\[
\vp(t,x)=\vp(t,x,E,\phi+\psi)
 = \langle f(x,E),\frac{t^2-1}{2}E+t(\phi+\psi)\rangle + F(x,E) - F(x,tE+\phi+\psi).
\]
%Assume that $E\in\cN$. We first show that
%\begin{equation}\label{eq:B3ineq_f_E}
%\int_{\Om}\langle f(x,E),E\rangle\, dx>0.
%\end{equation}
%Indeed, suppose that the above inequality does not hold. Then $\langle f(x,E),E\rangle=0$ a.e.\ in $\Om$, hence $F(x,E)=0$ a.e.\ in $\Om$. Since $F$ is of $\cC^1$ class in $E$, so $F$ attains a local minimum in $E$ and therefore $f(x,E)=0$ a.e.\ in $\Om$. As a consequence we obtain for $E = v^+ + \tv + w$
%$$
%0 = J'(E)[v^+] = Q(v^+),\quad
%0 = J'(E)[\tilde v] = Q(\tilde v),\quad
%0 = J'(E)[w] = \int_{\Om} \langle V(x) w,w\rangle\,dx,
%$$
%which implies $v = \tv \in \cV_0$ and $w=0$, hence $E\in\cV_0$. Now from
%$\int_\Om F(x,v_0)=0$ we deduce $J(E)=0$, contradicting $E\in\cN$.
%
Assumption (F5)(i) yields $\vp(t,x)=\vp(t,x,E,\phi+\psi)\le0$ for any $t\ge0$, a.e.\ $x\in\Om$. If $Q(\phi) < 0$ or $\int_\Om\langle V(x)\psi,\psi\rangle\,dx > 0$ then \eqref{eq:B3check} holds. If neither of these strict inequalities hold then $\phi\in\cV_0$ and $\psi=0$. In that case (F5)(ii) implies $\int_\Om\vp(t,x)\,dx < 0$.
\end{proof}

Now we recall the Nehari-Pankov manifold \eqref{eq:NehariCharact} for $\fJ$ given by
\[
\begin{aligned}
\cN &:= \{E\in X\setminus (\tcV\oplus\cW):\,J(E)>0,\,\fJ'(E)[E]=0\\\nonumber
    &\hspace{2cm}
            \hbox{ and }\fJ'(E)[\phi+\psi]=0\hbox{ for any }\phi\in\tcV,\,\psi\in\cW\}.
\end{aligned}
\]

Next we show that $\fJ$ satisfies the $(PS)_c^\cT$ condition on $\cN$ and on $\cM$ provided (F5) or (F6)-(F7) hold.

\begin{Lem}\label{lemmaconv2}
If (F5) holds then $J$ satisfies the $(PS)_c^\cT$ condition on $\cN$. If (F6)-(F7) hold then $J$ satisfies the $(PS)_c^\cT$ condition on $\cM$.
\end{Lem}

\begin{proof} Suppose (F5) holds and let $(E_n)_n\in\cN$ be a $(PS)_c$-sequence for $\fJ$ for some $c>0$, i.e.\
$$
\fJ(E_n)\to c \text{ and }\fJ'(E_n)\to0.
$$
Using (L2) and \eqref{eqest3} instead of \cite[(F4)]{BartschMederski1}, the proof follows from similar arguments as in \cite[Lemma 5.3]{BartschMederski1}.

Now assume that (F6)-(F7) holds and let $E_n=v_n+w_n\in\cM$ be a $(PS)_c$-sequence for $\fJ$. We need to show that $E_n\cTto E_0$ in $X$ for some $E_0\in X$ along a subsequence. Using (F7) we obtain
\[
\begin{aligned}
J(E_n)-\frac{1}{2}J'(E_n)(E_n)
 &=\int_{\Om}\frac{1}{2}\langle f(x,v_n+w_n),v_n+w_n\rangle -F(x,v_n+ w_n)\, dx\\
 &\geq \left(\frac{\ga}{2}-1\right)\int_{\Om}F(x,v_n+w_n)\, dx
\end{aligned}
\]
and
\[
\begin{aligned}
&J(E_n)-\frac{1}{\ga}J'(E_n)(E_n)\\
&\hspace{1cm}
 \geq \left(\frac12-\frac1\ga\right)
      \left(\int_{\Om}\langle\mu(x)^{-1}\curlop v_n,\curlop v_n\rangle\,dx
      - \int_\Om \langle V(x)(v_n+w_n),v_n+w_n\rangle\,dx\right)
\end{aligned}
\]
The above inequalities, \eqref{eqest3}, (L1) and the H\"older inequality imply that
\[
2J(E_n)-\left(\frac12-\frac{1}{\ga}\right)J'(E_n)(E_n)
 \geq \left(\frac12-\frac1\ga\right)\|v_n\|_\V^2-C_1|v_n+w_n|_p^2
       +\left(\frac\ga2-1\right)d'|v_n+w_n|^p_{p}
\]
for some constant $C_1>0$. Suppose that $|v_n+w_n|_{p}\to\infty$ as $n\to\infty$. Then for sufficiently large $n$ we have
\begin{eqnarray}\label{eqproofPS1}
2J(E_n)-\Big(\frac12-\frac{1}{\gamma}\Big)J'(E_n)(E_n)
&\geq&
\Big(\frac{1}{2}-\frac{1}{\gamma}\Big)\|v_n\|_\V^2
+\frac12\Big(\frac{\gamma}{2}-1\Big)d'|v_n+w_n|^p_{p}.
\end{eqnarray}
Note that
$\W$ is a closed subspace of $L^p(\Omega,\R^3)$,
$\cl\V\cap\W=\{0\}$ and therefore there is a continuous
projection of $\cl \V\oplus \W$ onto $\W$ in $L^p(\Omega,\R^3)$.
Hence there is a constant $C_2>0$ such that
$|w|_{p}\leq C_2|v+w|_{p}$
for any $v\in\V$ and $w\in \W$. Then \eqref{eqproofPS1} implies that $\|v_n\|_\V$ and $|w_n|_p$ are bounded which contradicts $|v_n+w_n|_{p}\to\infty$. Therefore $|v_n+w_n|_{p}$ must be bounded. By (F2), (F3), for any $\eps>0$ there is $c_\eps>0$ such that
for sufficiently large $n$
\begin{eqnarray*}
\|v_n\|^2_{\V}+\int_{\Om}\langle V(x)v_n,v_n\rangle\,dx&=&J'(E_n)(v_n)
-\int_{\Omega}\langle f(x,v_n+w_n),v_n\rangle\, dx\\
&\leq& \|v_n\|_{\V}+\int_{\Omega}(\eps|v_n+w_n|+c_{\eps}|v_n+w_n|^{p-1})|v_n|\, dx.
%&\leq& \|u_n\|_{1}+\eps\int_{\Omega} |u_n+\nabla w_n||u_n|\; dx
%+c_{\eps}\int_{\Omega} |u_n+\nabla w_n|^{p-1}|u_n|\; dx.
\end{eqnarray*}
Since $|v_n+w_n|_p$ is bounded, then the H\"older inequality and the Sobolev embeddings give
\begin{equation}\label{ineqPS3}
\|v_n\|^2_{\V}+\int_{\Om}\langle V(x)v_n,v_n\rangle\,dx
 \leq C_3\|v_n\|_{\V}
\end{equation}
for some constant $C_3>0$. Note that the H\"older inequality implies
$$
\int_{\Om}\langle V(x)v_n,v_n\rangle\,dx
 \leq \int_{\Om}\langle V(x)(v_n+w_n),v_n+w_n\rangle\,dx\leq C_4 |v_n+w_n|_p^2
$$
for some constant $C_4>0$. Therefore by \eqref{ineqPS3} we obtain that
$\|v_n\|_{\V}$ is bounded. In view of (L2), $|v_n|_{p}$ is bounded and then
$|w_n|_{p}$ is bounded.
Therefore $E_n=v_n+w_n$ is bounded in $X$ and we may assume, up to a subsequence,
$$
v_n \weakto v_0\hbox{ in }\V,\;v_n \to v_0 \hbox{ in }L^p(\Om,\R^3)
 \hbox{ and }w_n\weakto w_0\hbox{ in }\W
$$
for some $(v_0,w_0)\in\V\times\W$. Note that
\begin{eqnarray*}
J'(v_n,w_n)[v_n-v_0,0]&=&
 \|v_n-v_0\|_\V^2+\int_\Om\langle\mu^{-1}(x) \curlop v_0,\curlop (v_n- v_0)\rangle\,dx\\
&&      +\int_\Om\langle V(x)(v_n+w_n),v_n-v_0\rangle \,dx
     -\int_\Om\langle f(x,v_n+w_n),v_n-v_0\rangle \,dx.
\end{eqnarray*}
Since $(v_n)_{n}$ is bounded in $\V$, $v_n\to v_0$ in $L^2(\Om,\R^3)$ and
$(f(x,v_n+\nabla w_n))_{n}$ is bounded in $L^{\frac{p}{p-1}}(\Om,\R^3)$ we deduce
$\|v_n-v_0\|_\V\to 0$.
\end{proof}

\begin{altproof}{Proposition~\ref{prop:0isolated} b)}
If (F5) holds, then we easily conclude from the fact $\inf_{\cN}J>0$; see Theorem \ref{ThLink1}~a).

Suppose (F6)-(F8) hold, and assume by contradiction that there exists a sequence of nontrivial solutions $E_n=v_n+w_n \in\cV \oplus \cW$ such that
\[
J(E_n) = \frac12\int_\Om\langle\mu(x)^{-1}\curlop E_n,\curlop E_n\rangle\, dx
          - \frac12\int_\Om \langle V(x)E_n, E_n\rangle\, dx - \int_\Om F(x,E_n)\,dx \to 0.
\]
Then clearly $(E_n)$ is a Palais-Smale sequence in $\cM$ at level $0$. Now Lemma~\ref{lemmaconv2} implies $E_n\cTto E_0=v_0+w_0$ for some $E_0=v_0+w_0\in X$. Then as in the proof of Theorem~\ref{ThLink2} we get $E_n=m(v_n^+)\to m(v_0^+)=E_0$, so $E_0$ is a critical point of $J$. From
\[
o(1) = J(E_n) = \int_\Om \frac12\langle f(x,E_n),E_n\rangle - F(x,E_n)\,dx
 \ge \left(\frac\ga2-1\right) \int_\Om F(x,E_n)\,dx \ge 0\
\]
it follows that $\int_\Om F(x,E_n)\,dx \to 0$, so $\int_\Om F(x,E_0)\,dx=0$, hence $F(x,E_0(x))=0$ and $f(x,E_0(x))=0$ for a.e.\ $x\in\Om$. This implies
\[
0 = J'(E_n)[v_0^+] \to Q(v_0^+)
\]
which yields $v_0^+=0$. Similarly we obtain $w=0$ and finally $E_0\in\cV_0$. Now (F6)(ii) implies $E_0=0$.

Using $J'(E_n)[E_n]=0$, (F8) and (B4) we obtain
\begin{eqnarray*}
\|v_n^+\|^2
&=&  -Q(\tv_n)+
\int_\Om \langle V(x)(v^+_n+w_n, v^+_n+w_n\rangle\, dx + \int_\Om \langle f(x,E_n),E_n\rangle\,dx\\
&\leq& \Big(1-\frac\eta2\Big)\Big(-Q(\tv_n)+
\int_\Om \langle V(x)(v^+_n+w_n, v^+_n+w_n\rangle\, dx\Big)\\
&&+\eta\Big(-\frac12Q(\tv_n)+\frac12
\int_\Om \langle V(x)(v^+_n+w_n, v^+_n+w_n\rangle\, dx + \int_\Om F(x,E_n)\,dx\Big)\\
&=&\Big(1-\frac\eta2\Big)\Big(-Q(\tv_n)+
\int_\Om \langle V(x)(v^+_n+w_n, v^+_n+w_n\rangle\, dx\Big)+ \eta I(E_n)\\
&\leq& \Big(1-\frac\eta2\Big)\Big(-Q(\tv_n)+
\int_\Om \langle V(x)(v^+_n+w_n, v^+_n+w_n\rangle\, dx\Big)+ \eta I(v_n^+)\\
&\leq &\int_\Om \langle V(x)v^+_n, v^+_n\rangle\, dx+\eta\int_\Om F(x,v_n^+)\,dx.
\end{eqnarray*}
Therefore by \eqref{NormInU0} and (F2)-(F3), for any $\eps>0$ there exists $c_\eps>0$ such that
$$
\Big(1-\frac{1}{\lambda_m}\Big)\|v_n^+\|_{\cV}^2
 \leq Q(v_n^+)
 \leq \eta\int_\Om F(x,v_n^+)\,dx\leq \eps\eta |v_n^+|_2^2 +c_{\eps}\eta|v_n^+|_p^p
 \leq \eps \eta C\|v_n^+\|_{\cV}^2 +c_{\eps}\eta C\|v_n^+\|^p_{\cV}
$$
for some constant $C>0$. This contradicts $v_n^+\to v_0^+=0$.
\end{altproof}

\begin{altproof}{Theorem~\ref{thm:main}}
As a consequence of Lemmas~5.2-5.5 we may apply Theorem~\ref{ThLink1} in case (F5) holds, and Theorem~\ref{ThLink2} in case (F6)-(F7) holds. If (F5) holds then the solution is automatically a least energy solution being the minimizer of $J$ on the Nehari-Pankov manifold $\cN$. If (F6)-(F8) hold then the existence of a least energy solution is an immediate consequence of the $(PS)_c^\cT$ condition on $\cM$ and of Proposition~\ref{prop:0isolated}. Indeed, take a sequence of nontrivial critical points $(E_n)$ such that
$$J(E_n)\to c:=\inf\{J(E): E\in X\setminus\{0\}\hbox{ and } J'(E)=0\}>0$$
as $n\to\infty$. Since $(E_n)\subset \cM$, then by Lemma \ref{lemmaconv2} we find $E_0$ such that passing to a subsequence $E_n\cTto E_0$. In view of (A2) one has $J(E_0)\leq c$. On the other hand by (B4) we get $J(E_0)\geq c$. Thus $E_0$ is a least energy solution.

\end{altproof}

\section{Proof of Theorem~\ref{thm:sym1}}\label{sec:proof-sym}

Since $\Om$ is invariant under $G = O(2)\times{1}\subset O(3)$ we can define an action of $g \in G$ on $E\in L^2(\Om,\R^3)$ by setting
\[
(g*E)(x) := g\cdot E(g^{-1}x).
\]

\begin{Prop}\label{prop:sym}
If (S) holds then the action of $G$ on $X$ is isometric and leaves $\cV$ and $\cW$ invariant. Moreover, $J$ is invariant.
\end{Prop}

\begin{proof}
Clearly $G$ defines an isometric action on every $L^q(\Om,\R^3)$, in particular on $L^2(\Om,\R^3)$. A direct computation shows that $\curlop(g*E) = g*(\curlop E)$ holds for every $E\in\cC^1(\Om,\R^3)$. Since $\cC^1(\Om,\R^3)$ is dense in $H_0(\curl;\Om)$ it follows that $G$ induces an isometric action on $H_0(\curl;\Om)$, and $\curlop(g*E) = g*(\curlop E)$ holds for $E\in H_0(\curl;\Om)$ in the distributional sense. It also follows that $G$ induces an isometric action on $W^p(\curl;\Om)$. In order to see that $\cV$ is invariant we choose $v\in\cV$, $g\in G$ and $\varphi\in \cC^\infty_0(\Om,\R^3)$ with $\curlop\varphi=0$, and we compute:
\[
\begin{aligned}
&\int_\Om\langle V(x)(g*v)(x),\varphi(x)\rangle\,dx
 = \int_\Om\langle V(x)g\cdot v(g^{-1}x),\varphi(x)\rangle\,dx
 = \int_\Om\langle g\cdot V(x)v(g^{-1}x),\varphi(x)\rangle\,dx\\
&\hspace{1cm}
 = \int_\Om\langle V(x)v(g^{-1}x),g^{-1}\cdot\varphi(x)\rangle\,dx
 = \int_\Om\langle V(y)v(y),g^{-1}\cdot\varphi(gy)\rangle\,dy\\
&\hspace{1cm}
 = \int_\Om\langle V(y)v(y),g^{-1}*\varphi(y)\rangle\,dy
 = 0
\end{aligned}
\]
Here we used that $g$ commutes with every $V(x)$, that $g\in G$ is orthogonal, that $V$ is invariant with respect to the action of $G$ on $\Om$, that $\curlop(g^{-1}*\varphi)=g^{-1}*(\curlop\varphi)=0$, and that $v\in\cV$. It follows that $g*v\in\cV$. Clearly we also have
\[
\int_\Om\langle V(x)(g*E)(x),(g*E)(x)\rangle\,dx
 = \int_\Om\langle V(x)E(x),E(x)\rangle\,dx
\]
so that $\|g*v\|_\cV=\|v\|_\cV$ for $v\in\cV$. In a similar but easier way one sees that $G$ leaves $\cW$ invariant and preserves the norm.

In order to prove the invariance of $J$ with respect to the action of $G$ we use that $g\in G$ commutes with each $\mu(x)$, and that $\mu$ is $G$-invariant:
\[
\begin{aligned}
&\int_\Om\langle\mu(x)^{-1}\curlop(g*E)(x),\curlop(g*E)(x)\rangle\,dx\\
&\hspace{1cm}
 = \int_\Om\langle\mu(x)^{-1}g\cdot(\curlop E)(g^{-1}x),g\cdot\curlop E(g^{-1}x)\rangle\,dx\\
&\hspace{1cm}
 = \int_\Om\langle g\cdot\mu(x)^{-1}(\curlop E)(g^{-1}x),g\cdot\curlop E(g^{-1}x)\rangle\,dx\\
&\hspace{1cm}
 = \int_\Om\langle \mu(x)^{-1}(\curlop E)(g^{-1}x),\curlop E(g^{-1}x)\rangle\,dx\\
&\hspace{1cm}
 = \int_\Om\langle \mu(x)^{-1}(\curlop E)(x),\curlop E(x)\rangle\,dx
\end{aligned}
\]
Clearly we also have
\[
\int_\Om F(x,g*E)\,dx = \int_\Om F(x,E)\,dx.
\]
It follows that $J$ is invariant.
\end{proof}

Let $X^G=\cV^G\oplus\cW^G$ consist of all $G$-equivariant vector fields. By the principle of symmetric criticality, a critical point of the constrained functional $J|_{X^G}$ is a critical point of $\fJ$. Observe that
$$\cW_1:=\{w=\nabla \phi: \phi\in W^{1,p}_0(\Om)\}$$ is a closed subspace of $\cW$ and $\cW=\cK\oplus\cW_1$, where
\begin{eqnarray*}
\cK&:=&\left\{u\in \cW: \int_\Om\langle V(x) u,w\rangle\,dx=0
        \text{ for any }w\in \cW_1\right\}\\
    &=& \{u\in \cW: \div(V(x)u)=0 \}.
\end{eqnarray*}
It is easy to see that $\cK$ and $\cW_1$ are $G$-invariant. Let
\[
\cV_1=\cV\oplus\cK=\{E\in X:\div(V(x)E)=0\}
\]
and let $\cV_1^G=\V^G\oplus\cK^G$ consist of all $G$-equivariant vector fields as above.

We need the following lemma.
\begin{Lem}\label{LemDecomposition}
If (S) holds then any $E\in X^G$ has a unique decomposition $E=E_\tau+E_\rho+E_\zeta$ with summands of the form
\begin{equation*}
E_\tau(x)
 = \al(r,x_3)\begin{pmatrix}-x_2\\x_1\\0\end{pmatrix},\;
E_\rho(x)
 = \be(r,x_3)\begin{pmatrix}x_1\\x_2\\0\end{pmatrix},\;
E_\zeta(x)
 = \ga(r,x_3)\begin{pmatrix}0\\0\\1\end{pmatrix},
\end{equation*}
where $r=\sqrt{x_1^2+x_2^2}$. If $E\in\cV_1^G$ then $E_\tau,E_\rho+E_\zeta\in\cV_1^G$. If $E\in\cW_1^G$ then $E_\tau=0$. Moreover
\begin{equation}\label{eq:S-isom1}
\begin{aligned}
&\langle\mu(x)^{-1}\curlop  E_\rho(x),\curlop E_\tau(x)\rangle
 = \langle\curlop  E_\rho(x),\mu(x)^{-1}\curlop E_\tau(x)\rangle\\
&\hspace{1cm}
  =\langle\curlop E_\tau(x),\mu(x)^{-1}\curlop E_\zeta(x)\rangle
  = \langle\mu(x)^{-1}\curlop E_\tau(x),\curlop E_\zeta(x)\rangle = 0
\end{aligned}
\end{equation}
for a.e.\ $x\in\Om$.
\end{Lem}

\begin{proof}
The decomposition has been constructed in \cite[Lemma~1]{BenForAzzAprile} for vector fields $E\in\cD^1(\R^3)^G$. It extends immediately to $\cV_1^G$ and $\cW_1^G$.
%From $\nu\times E = 0$ on $\pa\Om$ we deduce $\al=0$ on $\pa\Om$ and $E_\tau \in H^1_0(\Om,\R^3)$ for $E\in\cV_1^G\subset H^1(\Om,\R^3)G.
Assumption (S) implies that $V$ depends only on $(r,x_3)$ and that $\div(V(x)v_\tau)=0$. Thus $v_\tau \in \V^G_1$ and $v_\rho+ v_\zeta=v-v_\tau\in\V^G_1$.
\end{proof}

\begin{altproof}{Theorem \ref{thm:sym1}}
%Observe that $\V$ coincides with $\V_1$ given by \eqref{eq:DefOfV_1}.
In view of Lemma \ref{LemDecomposition} the maps
\[
S_1:X^G \to X^G, \quad S_1(E_\tau+E_\rho+E_\zeta) := E_\tau-E_\rho-E_\zeta
\]
and
\[
S_2:X^G \to X^G, \quad S_2(E_\tau+E_\rho+E_\zeta, w) := -E_\tau+E_\rho+E_\zeta
\]
are well-defined linear isometries, and $S_1^2=S_2^2=\id$. It is easy to see that $J$ is invariant under $S_2$, provided (S) holds, of course. Moreover, if $F$ is in addition even then $J$ is also invariant under $S_1$. By the principle of symmetric criticality it is sufficient to find critical points of $J$ constrained to either
\[
(X^G)^{S_1} := \{E\in X^G: S_1(E)=E\} = \{E\in X^G: E=E_\tau\}\subset\cV_1,
\]
or to
\[
(X^G)^{S_2} := \{E\in X^G: S_2(E)=E\} = \{E\in X^G: E=E_\rho+E_\zeta\}.
\]
This can be done with the methods from Section~\ref{sec:proof-main} using Theorem~\ref{ThLink1} and Theorem \ref{ThLink2}. Observe that in Theorem \ref{thm:sym1}~a) we do not assume (L2) because $(X^G)^{S_1} = \{E\in X^G: E=E_\tau\} \subset H^1_0(\Om,\R^3)$ embeds compactly into $L^p(\Om,\R^3)$.
\end{altproof}

{\bf Acknowledgement.} We thank the reviewer for his/her careful reading.

%\newpage
{\sc Address of the authors:}\\[1em]
\parbox{8cm}{Thomas Bartsch\\
 Mathematisches Institut\\
 Universit\"at Giessen\\
 Arndtstr.\ 2\\
 35392 Giessen\\
 Germany\\
 Thomas.Bartsch@math.uni-giessen.de}%\\[1em]
\parbox{10cm}{
\vspace{5mm}
 Jaros\l aw Mederski\\
 Nicolaus Copernicus University \\
 Faculty of Mathematics and Computer Science\\
 ul.\ Chopina 12/18\\
 87-100 Toru\'n\\
 Poland\\
 jmederski@mat.umk.pl\\
 }


\begin{thebibliography}{99}
\baselineskip 2 mm

\bibitem{Amrouche}  C. Amrouche, C. Bernardi, M. Dauge, V. Girault:
{\em Vector potentials in three-dimensional non-smooth domains},
Math. Methods Appl. Sci. {\bf 21} (1998), no. 9, 823--864.

\bibitem{BenForAzzAprile} A. Azzollini, V. Benci, T. D'Aprile, D. Fortunato:
{\em Existence of Static Solutions of the Semilinear Maxwell Equations},
Ric. Mat. {\bf 55} (2006), no. 2, 283--297.

\bibitem{Ball2012} J. M. Ball, Y. Capdeboscq, B. Tsering-Xiao:
{\em On uniqueness for time harmonic anisotropic Maxwell's equations with piecewise regular coefficients},
Math. Models Methods Appl. Sci. {\bf 22} (2012), no. 11, 1250036, 11 pp.

\bibitem{BartschDing} T. Bartsch, Y. Ding:
{\em Deformation theorems on non-metrizable vector spaces and applications to critical point theory},
Mathematische Nachrichten {\bf 279} (2006), no. 12, 1267--1288.

\bibitem{Bartsch:2014} T. Bartsch, T. Dohnal, M. Plum, W. Reichel:
{\em Ground States of a Nonlinear Curl-Curl Problem in Cylindrically Symmetric Media}, Nonlinear Differ. Equ. Appl. (2016) 23:52, DOI 10.1007/s00030-016-0403-0.

\bibitem{BartschMederski1} T. Bartsch, J. Mederski:
{\em Ground and bound state solutions of semilinear time-harmonic Maxwell equations in a bounded domain},
Arch. Rational Mech. Anal., Vol. 215 (1), (2015), 283--306.

\bibitem{BartschMederskiSurvey} T. Bartsch, J. Mederski: {\em Nonlinear time-harmonic Maxwell equations in domains}, J. Fixed Point Theory Appl.
DOI 10.1007/s11784-017-0409-1.

\bibitem{BenFor} V. Benci, D. Fortunato:
{\em Towards a unified field theory for classical electrodynamics},
Arch. Rational Mech. Anal. {\bf 173} (2004), 379--414.

\bibitem{BenciRabinowitz} V. Benci, P. H. Rabinowitz:
{\em Critical point theorems for indefinite functionals},
Invent. Math. {\bf 52} (1979), no. 3, 241--273.

\bibitem{BuffaAmmariNed} A. Buffa, H. Ammari, J. C. Nédélec:
{\em Justification of Eddy Currents Model for the Maxwell Equations},
SIAM J. Appl. Math., {\bf 60} (5), (2000), 1805--1823.

%\bibitem{Corvellec-Degiovanni-Marzocchi:1993} J.-N. Corvellec, M. Degiovanni, and M. Marzocchi: {\em Deformation properties for continuous functionals and critical point theory}, Topol. Methods Nonlin. Anal. {\bf 1} (1993), no. 1, 151--171.

\bibitem{CostabelDN1999} M. Costabel, M. Dauge, S. Nicaise:
{\em Singularities of Maxwell interface problems},
Math. Model. Numer. Anal., {bf 33}, (1999), 627--649.

\bibitem{DAprileSiciliano} T. D'Aprile, G. Siciliano:
{\em Magnetostatic solutions for a semilinear perturbation of the Maxwell equations},
Adv. Differential Equations {\bf 16} (2011), no. 5--6, 435-466.

\bibitem{DingBook} Y. Ding:
{\em Variational Methods for Strongly Indefinite Problems}, Interdisciplinary Mathematical Sciences {\bf 7},
World Scientific Publishing 2007.

\bibitem{Doerfler} W. D\"orfler, A. Lechleiter, M. Plum, G. Schneider, C. Wieners:
{\em Photonic Crystals: Mathematical Analysis and Numerical Approximation},
Springer Basel 2012.

\bibitem{GilbargTrudinger} D. Gilbarg, N.S. Trudinger,
{\em Elliptic partial differential equations of second order},
Springer-Verlag, Berlin, 2001.

\bibitem{Hiptmair} R. Hiptmair:
{\em Finite elements in computational electromagnetism},
Acta Numerica {\bf 11}, (2002), 237--339.

\bibitem{KirschHettlich} A. Kirsch, F. Hettlich:
{\em The Mathematical Theory of Time-Harmonic Maxwell's Equations: Expansion-, Integral-, and Variational Methods},
Springer 2015.

%\bibitem{KryszSzulkin} W. Kryszewski, A. Szulkin: {\em Generalized linking theorem with an application to semilinear Schr\"odinger equation}, Adv. Diff. Eq. 3 (1998), 441--472.

\bibitem{Leis68} R. Leis:
{\em Zur Theorie elektromagnetischer Schwingungen in anisotropen inhomogenen Medien},
Math. Z. {\bf 106 } (1968), 213--224.

\bibitem{MederskiENZ} J. Mederski:
{\em Ground states of time-harmonic semilinear Maxwell equations in $\R^3$ with vanishing permittivity},
Arch. Rational Mech. Anal. 218 (2), (2015), 825--861.

\bibitem{Monk} P. Monk:
{\em Finite Element Methods for Maxwell's Equations}, Oxford University Press 2003.

\bibitem{Okaji:2002} T. Okaji:
{\em Strong unique continuation property for time harmonic Maxwell equations},
J. Math. Soc. Japan {\bf 54} (2002), 89--122.

%\bibitem{Pankov} A. Pankov:
%{\em Periodic Nonlinear Schr\"odinger Equation with Application to Photonic Crystals},
%Milan J. Math. {\bf 73} (2005), 259--287.

\bibitem{Pauly:2014} D. Pauly:
{\em Hodge-Helmholtz decompositions of weighted Sobolev spaces in irregular exterior domains with inhomogeneous and anisotropic media},
Math. Meth. Appl. Sc {\bf 31} (2008), 1509--1543.

\bibitem{Picard01} R. Picard, N. Weck, K.-J. Witsch:
{\em Time-harmonic Maxwell equations in the exterior of perfectly conducting, irregular obstacles},
Analysis (Munich) 21 (2001), no. 3, 231--263.

\bibitem{Rabinowitz:1986} P. Rabinowitz:
{\em Minimax Methods in Critical Point Theory with Applications to Differential Equations}, CBMS Regional Conference Series in Mathematics, Vol. {\bf 65}, Amer. Math. Soc., Providence, Rhode Island 1986.

\bibitem{FundPhotonics} B.E.A. Saleh, M.C. Teich:
{\em Fundamentals of Photonics}, 2nd Edition, Wiley 2007.

%\bibitem{Struwe} M. Struwe:
%{\em Variational Methods}, Springer-Verlag 2008.

%\bibitem{Stuart91} C. A. Stuart:
%{\em Self-trapping of an electromagnetic field and bifurcation from the essential spectrum}, Arch. %Rational Mech. Anal. {\bf 113} (1991), no. 1, 65--96.

%\bibitem{Stuart04} C. A. Stuart: {\em Modelling axi-symmetric travelling waves in a dielectric with nonlinear refractive index}, Milan J. Math. {\bf 72} (2004), 107--128.

\bibitem{StuartZhou96} C.A. Stuart, H.S. Zhou:
{\em A variational problem related to self-trapping of an electromagnetic field},
Math. Methods Appl. Sci. {\bf 19} (1996), no. 17, 1397--1407.

%\bibitem{StuartZhou01} C.A. Stuart, H.S. Zhou:{\em Existence of guided cylindrical TM-modes in a homogeneous self-focusing dielectric}, Ann. Inst. H. Poincar\'e Anal. Non Lin\'eaire {\bf 18} (2001), no. 1, 69--96.

\bibitem{StuartZhou03} C.A. Stuart, H.S. Zhou:
{\em A constrained minimization problem and its application to guided cylindrical TM-modes in an anisotropic self-focusing dielectric},
Calc. Var. Partial Differential Equations {\bf 16} (2003), no. 4, 335--373.

\bibitem{StuartZhou05} C.A. Stuart, H.S. Zhou:
{\em Axisymmetric TE-modes in a self-focusing dielectric},
SIAM J. Math. Anal. {\bf 37} (2005), no. 1, 218--237.

\bibitem{StuartZhou10} C.A. Stuart, H.S. Zhou:
{\em Existence of guided cylindrical TM-modes in an inhomogeneous self-focusing dielectric}, Math. Models Methods Appl. Sci. {\bf 20} (2010), no. 9, 1681--1719.

%\bibitem{SzulkinWethHandbook} A. Szulkin, T. Weth:
%{\em The method of Nehari manifold.} Handbook of nonconvex analysis and applications, 597--632,
%Int. Press, Somerville, 2010.

\bibitem{SzulkinWeth} A. Szulkin, T. Weth:
{\em Ground state solutions for some indefinite variational problems},
J. Funct. Anal. {\bf 257} (2009), no. 12, 3802--3822.

\bibitem{Nie} W. Nie:
{\em Optical Nonlinearity: Phenomena, applications, and materials},
Adv. Mater. {\bf 5}, (1993), 520--545.

\bibitem{Vogelsang:1991} V. Vogelsang:
{\em On the strong unique continuation principle for inequalities of Maxwell type},
Math. Ann. {\bf 289} (1991), 285--295.

%\bibitem{Willem} M. Willem:
%{\em Minimax Theorems}, Birkh\"auser Verlag 1996.

\end{thebibliography}
\end{document}